\documentclass[11pt]{article}

\usepackage{amsmath}
\usepackage{amsthm}
\usepackage{amsfonts}

\newcommand{\mmp}{\mathbb{P}}

\newcommand{\od}{\overset{d}{=}}
\newcommand{\dod}{\overset{d}{\to}}
\newcommand{\tp}{\overset{P}{\to}}

\newcommand{\me}{\mathbb{E}}
\newcommand{\mr}{\mathbb{R}}
\newcommand{\mn}{\mathbb{N}}

\newcommand{\lin}{\underset{n\to\infty}{\lim}}

\newcommand{\lix}{\underset{x\to\infty}{\lim}}
\newcommand{\lit}{\underset{t\to\infty}{\lim}}

\newtheorem{thm}{Theorem}[section]
\newtheorem{lemma}[thm]{Lemma}

\theoremstyle{definition}

\theoremstyle{remark}
\newtheorem{rem}[thm]{Remark}
\begin{document}
\title{On the number of empty boxes in the Bernoulli sieve II}
%
\author{Alexander Iksanov\footnote{ Faculty of Cybernetics, National T.
Shevchenko University of Kiev, 01033 Kiev, Ukraine,\newline
e-mail: iksan@univ.kiev.ua}}
\maketitle
\begin{abstract}
\noindent  The Bernoulli sieve is the infinite ``balls-in-boxes"
occupancy scheme with random frequencies $P_k=W_1\cdots
W_{k-1}(1-W_k)$, where $(W_k)_{k\in\mn}$ are independent copies of
a random variable $W$ taking values in $(0,1)$. Assuming that the
number of balls equals $n$, let $L_n$ denote the number of empty
boxes within the occupancy range. In the paper we investigate
convergence in distribution of $L_n$ in the two cases which
remained open after the previous studies. In particular, provided
that $\me |\log W|=\me |\log (1-W)|=\infty$ and that the law of
$W$ assigns comparable masses to the neighborhoods of $0$ and $1$,
it is shown that $L_n$ weakly converges to a geometric law. This
result is derived as a corollary to a more general assertion
concerning the number of zero decrements of nonincreasing Markov
chains. In the case that $\me |\log W|<\infty$ and $\me |\log
(1-W)|=\infty$ we derive several further possible modes of
convergence in distribution of $L_n$. It turns out that the class
of possible limiting laws for $L_n$, properly normalized and
centered, includes normal laws and spectrally negative stable laws
with finite mean. While investigating the second problem we
develop some general results concerning the weak convergence of
renewal shot-noise processes. This allows us to answer a question
asked in \cite{Res}.

\end{abstract}

\section{Introduction}

Let $(T_k)_{k\in\mn_0}$ be a multiplicative random walk defined by
$$T_0:=1, \ \ T_k:=\prod_{i=1}^k W_i, \ \ k\in\mn,$$ where $(W_k)_{k\in\mn}$ are independent copies of a random
variable $W$ taking values in $(0,1)$. Let $(U_k)_{k\in\mn}$ be
independent random variables with the uniform $[0,1]$ law which
are independent of the multiplicative random walk. The {\it
Bernoulli sieve} is a random occupancy scheme in which `balls'
$U_k$'s are allocated over infinitely many `boxes' $(T_k,
T_{k-1}]$, $k\in\mn$. The scheme was introduced in \cite{Gne}.
Further investigations were made in \cite{slow, GneIksMar, GIM2,
GINR, GIR, Iks}. Since a particular ball falls in box $(T_k,
T_{k-1}]$ with probability
\begin{equation*}
P_k:=T_{k-1}-T_k=W_1W_2\cdots W_{k-1}(1-W_k),
\end{equation*}
the Bernoulli sieve is also the classical infinite occupancy
scheme \cite{GnePitHan, Karlin} with random frequencies
$(P_k)_{k\in\mn}$, where (abstract) balls are allocated over an
infinite array of (abstract) boxes $1,2,\ldots$ independently
conditionally given $(P_k)$ with probability $P_j$ of hitting box
$j$. Alternatively the Bernoulli sieve can be thought of as a
randomized variant of the leader election procedure which appears
if the law of $W$ is degenerate at some $x\in (0,1)$ (this may be
especially appropriate for the reader familiar with the analysis
of algorithms).

We will use the following notation for the moments
$$\mu:=\me |\log W| \ \ \text{and} \ \ \nu:=\me |\log(1-W)|$$
which may be finite or infinite. Assuming that the number of balls
equals $n$ denote by $K_n$ the number of occupied boxes, $M_n$ the
index of the last occupied box, and $L_n:=M_n-K_n$ the number of
empty boxes within the occupancy range. The present paper is a
contribution towards understanding the weak convergence of $L_n$.
With the account of the results obtained here and in some previous
works on the subject we can now draw an almost complete picture
(Remark \ref{en1} which discusses two cases where the weak
convergence of $L_n$ remains unsettled reveals what is hidden
behind the word 'almost'). Depending on the behavior of the law of
$W$ near the endpoints $0$ and $1$ the number of empty boxes can
exhibit quite a wide range of different asymptotics.

\noindent {\sc Case $\mu<\infty$ and $\nu<\infty$}: $L_n$
converges in distribution and in mean to some $L$ with proper and
nondegenerate law (Theorem 2.2(a) in \cite{GINR} and Theorem 3.3
in \cite{GIR}). Furthermore there is also convergence of all
moments (Theorem 20(b) in \cite{Negad}).

\noindent {\sc Case $\mu=\infty$ and $\nu<\infty$}: $L_n$
converges to zero in probability (Theorem 2.2(a) in \cite{GINR}).

\noindent {\sc Case $\mu<\infty$ and $\nu=\infty$}: There are
several possible modes of the weak convergence of $L_n$, properly
normalized and centered (see Theorem \ref{main7} of the present
paper).

\noindent {\sc Case $\mu=\infty$ and $\nu=\infty$}: The
asymptotics of $L_n$ is determined by the behavior of the ratio
$\mmp\{W\leq x \}/\mmp\{1-W\leq x\}$, as $x\downarrow 0$. When the
law of $W$ assigns much more mass to the neighborhood of $1$ than
to that of $0$ equivalently the ratio goes to $0$, $L_n$ becomes
asymptotically large. In this situation the weak convergence
result for $L_n$, properly normalized without centering, was
obtained in \cite{Iks} under a condition of regular variation. If
the roles of $0$ and $1$ are interchanged $L_n$ converges to zero
in probability (this follows from Theorem 7.1(i) in \cite{GIM2}
and Markov inequality). When the tails are comparable $L_n$ weakly
converges to a geometric distribution (see Theorem \ref{main2} of
the present paper).

Also it was known that whenever $L_n \dod L$, where $L$ is a
random variable with a proper and nondegenerate probability law,
the law of $L$ is mixed Poisson (Proposition 1.2 in \cite{Iks}),
and that $L_n$ has the geometric distribution with parameter $1/2$
when $W\od 1-W$ (Proposition 7.1 in \cite{GIM2}).

Throughout the paper ${\rm geom}(a)$ denotes a random variable
which has the geometric distribution (starting at zero) with
success probability $a$, i.e.,
$$\mmp\{{\rm geom}(a)=m\}=a(1-a)^m, \ \ m\in\mn_0,$$ and $\mathcal{N}(0,1)$ denotes a random variable which has the
standard normal distribution.

We are ready to state our first result which treats the case of
'comparable tails' when $\mu=\nu=\infty$.
\begin{thm}\label{main2}
Suppose $\mu=\infty$ and
\begin{equation}\label{2}
\lin {\me W^n\over \me(1-W)^n}=c\in (0,\infty).
\end{equation}
Then
\begin{equation}\label{3}
L_n\ \dod \ L\od {\rm geom}((c+1)^{-1}), \ \ n\to\infty.
\end{equation}
In particular, relation \eqref{3} holds whenever the tails are
comparable, i.e.,
\begin{equation}\label{5}
\underset{x\downarrow 0}{\lim}\,{\mmp\{1-W \leq
x\}\over\mmp\{W\leq x\}}=c.
\end{equation}
\end{thm}

The situation when $\mu<\infty$ and $\nu=\infty$ is covered by
Theorem \ref{main7} which is our second result. 
\begin{thm}\label{main7}
Suppose $\nu=\infty$, and the law of $|\log W|$ is non-lattice.
Set
\begin{equation*}\label{102} b_n:={1\over \mu}\int_{[1,n]}
{\psi(z)\over z}{\rm d}z,
\end{equation*}
where $\psi(s):=\me e^{-s(1-W)}$, $s\geq 0$.
\newline {\rm (a)} If $\sigma^2={\rm Var}\,(\log W)<\infty$ then, with
$a_n:=\sqrt{b_n}$, the limiting distribution of ${L_n-b_n\over
a_n}$ is standard normal.
\newline {\rm (b)} Assume that $\sigma^2=\infty$ and
\begin{equation}\label{domain0}
\int_{[0,x]} y^2 \mmp\{|\log W|\in {\rm d}y\} \ \sim \
\widetilde{\ell}(x), \ \ x\to\infty,
\end{equation}
for some $\widetilde{\ell}$ slowly varying at $\infty$. Let $c(x)$
be any positive function satisfying
$\lix\,x\widetilde{\ell}(c(x))/c^2(x)=1$ which implies that
$c(x)\sim x^{1/2}\ell^\ast(x)$, $x\to\infty$, for some $\ell^\ast$
slowly varying at $\infty$.
\newline {\rm (b1)}
If
\begin{equation}\label{555}
\lix \mmp\{|\log (1-W)|>x\}(\ell^\ast(x))^2=0
\end{equation}
then, with $a_n=\sqrt{b_n}$, the limiting distribution of
${L_n-b_n\over a_n}$ is standard normal. \newline {\rm (b2)}
Assume that
\begin{equation}\label{domain2}
\mmp\{|\log (1-W)|>x\} \ \sim \ \ell(x), \ \ x\to\infty,
\end{equation}
for some $\ell$ slowly varying at $\infty$, and that $$\lix
\mmp\{|\log (1-W)|>x\}(\ell^\ast(x))^2=\infty.$$ Then, with
$a_n:=\mu^{-3/2}c(\log n)\psi(n)$, the limiting distribution of
${L_n-b_n\over a_n}$ is standard normal.
\newline {\rm (c)} Assume that
\begin{equation}\label{domain1}
\mmp\{|\log W|>x\} \ \sim \ x^{-\alpha}\widetilde{\ell}(x), \ \
x\to \infty,
\end{equation}
for some $\widetilde{\ell}$ slowly varying at $\infty$ and
$\alpha\in (1,2)$. Let $c(x)$ be any positive function satisfying
$\lix \,x\widetilde{\ell}(c(x))/c^\alpha(x)=1$ which implies that
$c(x)\sim x^{1/\alpha}\ell^\ast(x)$, $x\to\infty$, for some
$\ell^\ast$ slowly varying at $\infty$.\newline {\rm (c1)} If
\begin{equation}\label{555555}
\lix \mmp\{|\log (1-W)|>x\}x^{2/\alpha-1}(\ell^\ast(x))^2=0,
\end{equation}
then, with $a_n=\sqrt{b_n}$, the limiting distribution of
${L_n-b_n\over a_n}$ is standard normal.\newline {\rm (c2)} Assume
that
\begin{equation}\label{domain3} \mmp\{|\log
(1-W)|>x\} \ \sim \ x^{-\beta}\ell(x), \ \ x\to\infty,
\end{equation}
for some $\beta\in [0, 2/\alpha-1]$ and some $\ell$ slowly varying
at $\infty$. In the case $\beta=2/\alpha-1$ assume additionally
that $$\lix
\mmp\{|\log(1-W)>x|\}x^{2/\alpha-1}(\ell^\ast(x))^2=\infty.$$ Then
$${L_n-b_n\over \mu^{-1-1/\alpha}c(\log n)\psi(n)} \ \dod \
\int_{[0,1]}v^{-\beta}{\rm d}Z(v),$$ where $(Z(v))_{v\in [0,1]}$
is an $\alpha$-stable L\'{e}vy process such that $Z(1)$ has
characteristic function
\begin{equation}\label{st1}
u\mapsto \exp\{-|u|^\alpha
\Gamma(1-\alpha)(\cos(\pi\alpha/2)+i\sin(\pi\alpha/2)\, {\rm
sgn}(u))\}, \ u\in\mr.
\end{equation}

\noindent Throughout one can take
$b_n^\prime:=\mu^{-1}\int_{[0,\,\log n]}\mmp\{|\log (1-W)|>x\}{\rm
d}x$ in place of $b_n$.
\end{thm}
\begin{rem}
The integrals $\int_{[0,1]}v^{-\beta}{\rm d}Z(v)$ appearing in the
theorem and also in formulae \eqref{101} and \eqref{102} are
understood to be equal to $Z(1)$ in the case $\beta=0$ and to be
defined by integration by parts formula
$$\int_{[0,1]}v^{-\beta}{\rm d}Z(v)=Z(1)+\beta\int_{[0,1]}v^{-\beta-1}Z(v){\rm d}v$$
in the case $\beta\in (0,1/\alpha)$ (when referring to formula
\eqref{101} we take $\alpha=2$). Note that the latter is
consistent with the standard definition of stochastic integrals
(with respect to semimartingales). It is known that $$\log \me
\exp\bigg({\rm i}t \int_{[0,1]}v^{-\beta}{\rm
d}Z(v)\bigg)=\int_{[0,1]}\log \me\exp \big({\rm
i}tv^{-\beta}Z(1)\big){\rm d}v, \ \ t\in\mr,$$ from which it
follows that the integral is indeed well-defined only if $\beta\in
[0,1/\alpha)$ and that
$$\int_{[0,1]}v^{-\beta}{\rm d}Z(v)\od (1-\alpha\beta)^{-1/\alpha}Z(1).$$
\end{rem}
\begin{rem}\label{en1}
Theorem \ref{main7} does not cover two interesting cases. Assume
that the standing assumptions of the theorem hold.\newline {\rm
Case (b3)}: Condition \eqref{domain0} holds, $\sigma^2=\infty$,
and
$$\mmp\{|\log (1-W)|>x\} \ \sim \ {d\over (\ell^\ast(x))^2}, \ \
x\to\infty,$$ for some $d>0$ and $\ell^\ast(x)$ defined in part
(b) of the theorem. \newline {\rm Case (c3)}: Condition
\eqref{domain1} holds, and
$$\mmp\{|\log (1-W)|>x\} \ \sim \ {dx^{1-2/\alpha}\over
(\ell^\ast(x))^2}, \ \ x\to\infty,$$ for some $d>0$ and
$\ell^\ast(x)$ defined in part (c) of the theorem.

Some partial results and discussion of the problems which arise in
these cases can be found in Remark \ref{en}.
\end{rem}
\begin{rem}
We conjecture that under the assumption $\mu<\infty$ the
conditions given in Theorem \ref{main7} and Remark \ref{en1} are
necessary and sufficient for the weak convergence of $L_n$,
properly normalized and centered.
\end{rem}

The rest of the paper is structured as follows. In Section
\ref{nond} we point out the set of conditions under which the
number of zero decrements of a nonincreasing Markov chain weakly
converges to a geometric law (Theorem \ref{main}). Theorem
\ref{main2} then follows as a particular case. Section \ref{ma} is
devoted to proving Theorem \ref{main7}. Some results derived in
Section \ref{ma} can be used to answer a question asked in
\cite{Res}. A detailed discussion of this is given in Section
\ref{re}. Some auxiliary facts are collected in the Appendix.

\section{Number of zero decrements of nonincreasing Markov
chains}\label{nond}

\subsection{Definitions}\label{non}

With $M\in\mn_0$ given and any $n\geq M$, $n\in\mn$, let
$I:=\big(I_k(n)\big)_{k\in\mn_0}$ be a {\it nonincreasing Markov
chain} with $I_0(n)=n$, state space $\mn$ and transition
probabilities
\begin{eqnarray*}
\mmp\{I_k(n)=j|I_{k-1}(n)=i\}& = &\pi_{i,j},\;\;i\geq M+1 \ \ \text{and either} \ M< j\leq i \ \text{or} \ M=j<i,\\
\mmp\{I_k(n)=j|I_{k-1}(n)=i\}& = &0,\;\;i< j,\\
\mmp\{I_k(n)=M|I_{k-1}(n)=M\}& = &1.
\end{eqnarray*}
Denote by
$$Z_n:=\#\big\{k\in\mn_0: I_k(n)-I_{k+1}(n)=0, I_k(n)>M\big\}$$
the number of zero decrements of the Markov chain before the
absorption. Assuming that, for every $M+1\leq i\leq n$,
$\pi_{i,\,i-1}>0$, the absorption at state $M$ is certain, and
$Z_n$ is a.s.\,finite.

Neglecting zero decrements of $I$ along with renumbering of
indices lead to a {\it decreasing Markov chain}
$J:=\big(J_k(n)\big)_{k\in\mn_0}$ with $J_0(n)=n$ and transition
probabilities
$$\widetilde{\pi}_{i,j}={\pi_{i,j}\over 1-\pi_{i,\,i}}, \ \ i>j\geq
M$$ (the other probabilities are the same as for $I$). The chain
$J$ visits a given state $k$ and the chain $I$ visits the state
$k$ for the first time with the same probability
$$g_{n,k}:=\sum_{m\geq
0}\mmp\{J_m(n)=k\}, \ \ k\leq n, k\in\mn.$$ Note that $g_{n,n}=1$
and that $g_{n,k}$ is the potential function of $J$.

Let $(R_j)_{M+1\leq j\leq n}$ be independent random variables such
that $R_j\od {\rm geom}(1-\pi_{j,j})$. Assuming the $R_j$'s
independent of the sequence of states visited by $J$ we may
identify $R_j$ with the time $I$ spends in the state $j$ provided
this state is visited. With this at hand $Z_n$ can be conveniently
represented as
\begin{equation}\label{rep}
Z_n\od \sum_{k\geq 0}R_{J_k(n)}1_{\{J_k(n)>M\}}.
\end{equation}

\subsection{Main result of the section}

Theorem \ref{main} given below proves that the number of zero
decrements of a nonincreasing Markov chain weakly converges to a
geometric law whenever the probability of delay at the present
state and that of transition to the absorption state are
asymptotically balanced, and the Markov chain has no 'stationary'
version. An interesting feature of this quite general result is
that its proof needs nothing beyond simple distributional
recurrence \eqref{15}.
\begin{thm}\label{main}
Assume that $\lin g_{n,k}=0$ for each $k\in\mn$, $\lin
\pi_{n,n}=0$ and
\begin{equation}\label{4}
\lin {\pi_{n,n}\over \pi_{n,M}}=c\in (0,\infty).
\end{equation}
Then
$$Z_n\ \dod \ Z\od {\rm geom}((c+1)^{-1}), \ \ n\to\infty.$$
\end{thm}

Theorem \ref{main} will be proved by the method of moments. To
this end, we have to possess some information about the moments of
integer orders of the limiting geometric law. The explicit
expressions are complicated and actually not needed. The moments
satisfy a simple recurrence which is sufficient for our needs.
\begin{lemma}\label{6}
Let $X\od {\rm geom}(a)$, $a>0$. The moments $m_k:=\me X^k$,
$k\in\mn$ can be recursively obtained via
\begin{equation}\label{7}
m_1=b, \ \ m_j=b\bigg(1+\sum_{i=1}^{j-1}{j \choose i}m_i\bigg), \
\ j=2,3,\ldots,
\end{equation}
where $b:=(1-a)/a$.
\end{lemma}
\begin{proof}
Let $(\zeta_k)_{k\in\mn}$ be independent Bernoulli random
variables with success probability $a$. Then $$X\od \inf\{k\in\mn:
\zeta_k=1\}-1=1_{\{\zeta_1=0\}}\big(1+(\inf\{k\in\mn\backslash\{1\}:
\zeta_k=1\}-1)\big)=:1_{\{\zeta_1=0\}}(1+X^\prime),$$ where
$X^\prime$ is independent of $\zeta_1$ and $X^\prime\od X$. The
latter implies $$\me X^j=(1-a)\me (1+X)^j, \ \ j\in\mn,$$ and
representation \eqref{7} follows.
\end{proof}

Now we are ready to prove Theorem \ref{main}. For notational
convenience we assume that $M=0$. For other $M$'s the argument is
the same.

Let $V_n$ denote the size of the last decrement. Then
\begin{equation}\label{5005}
\mmp\{V_n=k\}=g_{n,k}\widetilde{\pi}_{k,0}=g_{n,k}{\pi_{k,0}\over
1-\pi_{k,k}}, \ \ k=1,2,\ldots, n,
\end{equation}
and
\begin{equation}\label{8}
\lin \mmp\{V_n=k\}=0, \ \ \text{for each} \ k\in\mn.
\end{equation}

Since the geometric law is uniquely determined by its moments, it
suffices to prove that, for each $i\in\mn$, $\lin \me Z_n^i=\me
Z^i$. To this end, we will use the induction on $i$ and start with
the case $i=1$. Using representation \eqref{rep} and conditioning
on the first decrement of $J$ we deduce the distributional
equality
\begin{equation}\label{15}
Z_n \od \widehat{Z}_{J(n)}+R_n,
\end{equation}
where, for each $k\in\mn$, $\widehat{Z}_k$ is independent of both
$J(n):=J_1(n)$ and $R_n$, and has the same law as $Z_k$. Equality
\eqref{15} (or just \eqref{rep}) implies that
\begin{eqnarray*}
\me Z_n&=&\sum_{k=1}^n g_{n,k}\me
R_k\overset{\eqref{5005}}{=}\sum_{k=1}^n
\mmp\{V_n=k\}{\pi_{k,k}\over \pi_{k,0}}
\end{eqnarray*}
Recalling \eqref{8} and \eqref{4} and applying Lemma
\ref{Toeplitz1} to the last sum lead to the conclusion $\lin \me
Z_n=c=\me Z$.

Assume now that $\lin \me Z_n^i=\me Z^i$ for all $i\leq j-1$,
$i\in\mn$. We have to prove that $\lin \me Z_n^j=\me Z^j$. In view
of Lemma \ref{6} it suffices to check that $\lin \me Z_n^j=m_j$,
where $m_j$ satisfies \eqref{7} with $b=c$ and $m_i=\me Z^i$.
Using \eqref{15} yields $$\me Z_n^j=\me
Z^j_{J(n)}+\sum_{i=0}^{j-1} {j \choose i}\me Z^i_{J(n)}\me
R_n^{j-i}=: \me Z^j_{J(n)}+b_n,$$ or, equivalently, $$\me
Z_n^j=\sum_{k=1}^n g_{n,\,k}
b_k=\sum_{k=1}^n\mmp\{V_n=k\}{1-\pi_{k,\,k}\over
\pi_{k,\,0}}b_k.$$ In view of Lemma \ref{Toeplitz1} to finish the
proof it remains to show that
$$\lin {1-\pi_{n,\,n}\over \pi_{n,\,0}}b_n=c\bigg(1+\sum_{i=1}^{j-1}{j
\choose i}m_i\bigg)$$ or equivalently that, for $i\leq j-1$,
\begin{equation}\label{166}
\lin {1-\pi_{n,\,n}\over \pi_{n,\,0}}\me Z^i_{J(n)}\me
R_n^{j-i}=c\, \me Z^i \ \ \text{and} \ \ \lin {1-\pi_{n,\,n}\over
\pi_{n,\,0}}\me R_n^j=c.
\end{equation}
Applying Lemma \ref{6} with $a=1-\pi_{n,\,n}$ to the $R_n$'s we
conclude that $$\me R_n^{j-i} \ \sim \ \me R_n\ \sim \
\pi_{n,\,n}, \ \ n\to\infty.$$ Further, one can easily check that
$\lin b_n=0$, hence $$\lin \me Z_{J(n)}^i=\me Z^i, \ \ i\leq
j-1.$$ These two observations immediately establish \eqref{166}.
The proof is complete.

\subsection{Proof of Theorem \ref{main2}}

In this subsection we prove Theorem \ref{main2} by an application
of Theorem \ref{main}. To distinguish general (nonincreasing)
Markov chains in the previous subsections from the particular
Markov chain discussed below we mark all the quantities which
correspond to the latter with asterisk, for instance, $I\to
I^\ast$, $g_{n,k}\to g^\ast_{n,k}$ etc.

Now we present one more construction of the Bernoulli sieve which
highlights the connection with nonincreasing Markov chains. The
Bernoulli sieve can be realized as a random occupancy scheme in
which $n$ 'balls' are allocated over an infinite array of 'boxes'
indexed $1$, $2,\ldots$ according to the following rule. At the
first round each of $n$ balls is dropped in box $1$ with
probability $W_1$. At the second round each of the remaining balls
is dropped in box $2$ with probability $W_2$, and so on. The
procedure proceeds until all $n$ balls get allocated. Let
$I^\ast_k(n)$ denote the number of remaining balls (out of $n$)
after the $k$th round. Then $I^\ast:=(I^\ast_k(n))_{k\in\mn_0}$ is
a pattern of nonincreasing Markov chains described in Subsection
\ref{non} with $M=0$ and
\begin{equation}\label{18}
\pi^\ast_{i,j}={i\choose j}\me W^j(1-W)^{i-j}, \ \ j\leq i.
\end{equation}
It is plain that $L_n$ is the number of zero decrements of
$I^\ast$ before the absorption.

The assumption $\mu=\infty$ implies that $\lin g^\ast_{n,k}=0$,
for each $k\in\mn$. In the case that the law of $|\log W|$ is
nonlattice this fact was pointed out in formula (16) in
\cite{GINR}. The complementary case does not require any new proof
once one has noticed that the overshoot at point $x$ of a standard
random walk diverges to $+\infty$ in probability (as $x\to\infty$)
under the sole assumption that the step of the random walk has
infinite mean. In view of \eqref{18} $\pi^\ast_{n,n}=\me W^n\to
0$, as $n\to\infty$ (recall that the law of $W$ has no atom at
$1$), and condition \eqref{4} reduces to \eqref{2}. According to
Theorem \ref{main}, relation \eqref{3} holds.

Condition \eqref{5} is equivalent to
$$\underset{x\downarrow 0}{\lim}\, {\mmp\{|\log W|\leq x\}\over \mmp\{|\log (1-W)|\leq
x\}}=c.$$ Applying Lemma \ref{taub} with $\xi=|\log W|$ and
$\eta=|\log(1-W)|$ establishes implication \eqref{5} $\Rightarrow$
\eqref{2}, thereby completing the proof of Theorem \ref{main2}.

\section{Proof of Theorem \ref{main7}}\label{ma}
%


In the first (main) part of the proof we work with a poissonized
version of the Bernoulli sieve. Specifically we assume that the
balls are thrown at arrival times $(\tau_n)_{n\in\mn}$ of a unit
rate Poisson process $(\pi_t)_{t\geq 0}$. The quantity in focus is
then $L(t):=L_{\pi_t}$, where $(\pi_t)$ is independent of $(L_j)$.
At the last step of the proof we return to the original, fixed $n$
problem (this step is called {\em depoissonization}) and prove the
implication
$$ {L(t)-b(t)\over a(t)} \ \dod \ X, \ \ t\to\infty \Rightarrow \ {L_n-b(n)\over a(n)} \ \dod \ X, \ \ n\to\infty.$$

\noindent Set $$S_k:=-\log T_k=|\log W_1|+\ldots+|\log W_k| \ \
\text{and} \ \ \widehat{S}_k:=\widehat{S}_0+S_k, \ \ k\in\mn_0,$$
where $\widehat{S}_0$ is a random variable which is independent of
$(S_k)$ and has distribution
$$\mmp\{\widehat{S}_0\leq x\}=\mu^{-1}\int_{[0,x]}\mmp\{|\log
W|>y\}{\rm d}y, \ \ x\geq 0.$$ Define $$N(x):=\inf\{k\in\mn_0:
S_k>x\}=\#\{k\in\mn_0: S_k\leq x\} \ \ \text{and} \ \
\widehat{N}(x):=\#\{n\in\mn_0: \widehat{S}_n\leq x\}, \ \ x\geq
0,$$ and recall that $(N(x))_{x\geq 0}$ and
$(\widehat{N}(x))_{x\geq 0}$ are non-stationary and stationary
renewal processes, respectively. For later use, we recall (see
p.~55 in \cite{Gut} for the proof) that $N(x)$ enjoys the
following (distributional) subadditivity property
\begin{equation}\label{yyy}
N(x+y)-N(x)\overset{d}{\leq} N(y), \ \ x,y\geq 0.
\end{equation}
In the sequel we work with the following random variables
$$C(t):=\sum_{k\geq 0}\varphi(t-S_k)1_{\{S_k\leq t\}}=\int_{[0,\,t]}\varphi(t-x){\rm d}N(x), \ \ t\geq 0$$
and $$\widehat{C}(t):=\sum_{k\geq
0}\varphi(t-\widehat{S}_k)1_{\{\widehat{S}_k\leq
t\}}=\int_{[0,\,t]}\varphi(t-x){\rm d}\widehat{N}(x), \ \ t\geq
0,$$ where $\varphi(t):=\psi(e^t)$, $t\in\mr$, $\psi(t):=\me
e^{-t(1-W)}$, $t\geq 0$.

We show in Lemma \ref{au} that convergence in distribution of
$L(t)$ is completely determined by convergence in distribution of
$$L^\ast(t):=\sum_{k\geq
1}\exp(-te^{-S_{k-1}}(1-W_k))1_{\{S_{k-1}\leq \log t\}}.$$ The
Bernoulli sieve is governed by two sources of randomness:
randomness of the 'environment' $(W_k)$ and sampling variability
(i.e. the variability of the occupancy scheme with {\it
deterministic} frequencies obtained by conditioning on $(W_k)$).
Since $L^\ast(t)$ is a function of the environment alone we
conclude that the weak convergence of $L(t)$ ($L_n$) is completely
determined by the randomness of the environment, whereas the
influence of the sampling variability is negligible.

In its turn convergence in distribution of $L^\ast(t)$ is
determined either by that of $L^\ast(t)-C(\log t)$ or that of
$C(\log t)$, or that of both, and our main task is to find out
what is the extent of their interplay. In the cases (a), (b1) and
(c1) the contribution of $L^\ast(t)-C(\log t)$ dominates, whereas
in the cases (b2) and (c2) it is negligible in comparison with the
contribution of $C(\log t)$.\footnote{It seems that there are
situations (cases (b3) and (c3) introduced in Remark \ref{en1})
when contributions of both variables are significant, and both of
these determine the asymptotics of $L(t)$. See Remark \ref{en1}
and Remark \ref{en} for more details.}

We divide the proof of the theorem into several steps.

\noindent {\sc Step 1}. The purpose of this step is proving a
central limit theorem for $L(t)-C(\log t)$ (Lemma \ref{main6}). To
this end we first show that the asymptotic behavior of $L(t)$
coincides with that of $L^\ast(t)$.

In what follows we write that the family of random variables is
tight meaning that the family of laws of these random variables is
tight.
\begin{lemma}\label{au}
Whenever $\mu<\infty$ and the law of $|\log W|$ is non-lattice,
the families $\big(L(t)-L^\ast(t)\big)_{t\geq 1}$ and
$\big(C(t)-\widehat{C}(t)\big)_{t\geq 0}$ are tight.
\end{lemma}
\begin{proof}
Set $M(t):=M_{\pi_t}$ and $K(t):=K_{\pi_t}$. These are the index
of the last occupied box and the number of occupied boxes in the
poissonized version of the Bernoulli sieve, respectively. Clearly,
$L(t)=M(t)-K(t)$.

\noindent {\sc Fact 1}: The family $\big(M(t)-N(\log t))_{t\geq
1}$ is tight.

We use the representation $M(t)=N(|\log U_{1,\pi_t}|)$, where
$U_{1,n}:=\underset{1\leq j\leq n}{\min}\, U_j$. It is well-known
that $|\log U_{1,n}|-\log n \dod G$, $n\to\infty$, where $G$ is a
random variable with the standard Gumbel distribution. Since
$(\pi_t)$ is independent of $U_{1,n}$ we also have $|\log
U_{1,\pi_t}|-\log \pi_t \dod G$, $t\to\infty$. By noting that
$\log \pi_t-\log t \tp 0$, $t\to\infty$ we finally conclude that
$|\log U_{1,\pi_t}|-\log t \dod G$, $t\to\infty$. Using
\eqref{yyy} along with independence of $\big(N(x)\big)$ and
$U_{1,\,\pi_t}$ we obtain
\begin{eqnarray*}
M(t)-N(\log t)&\leq& \big(N(|\log U_{1,\,\pi_t}|)-N(\log
t))\big)1_{\{|\log U_{1,\,\pi_t}|\geq \log
t\}}\\&\overset{d}{\leq}& N(|\log U_{1,\,\pi_t}|-\log
t\big)1_{\{|\log U_{1,\,\pi_t}|\geq \log t\}} \ \dod \
N(G)1_{\{G\geq 0\}}, \ \ t\to\infty.
\end{eqnarray*}
Similarly
\begin{eqnarray*}
M(t)-N(\log t)&\geq& -\big(N(\log t)-N(|\log
U_{1,\pi_t}|)\big)1_{\{|\log U_{1,\pi_t}|< \log
t\}}\\&\overset{d}{\geq}& -N\big(\log t-|\log
U_{1,\pi_t}|\big)1_{\{|\log U_{1,\pi_t}|< \log t\}} \ \dod \
-N(-G)1_{\{G<0\}}, \ \ t\to\infty.
\end{eqnarray*}

\noindent {\sc Fact 2}: The family $\big(K(t)-\me
(K(t)|(W_k))\big)_{t\geq 0}$ is tight.

This was proved in formula (28) in \cite{GneIksMar}.

\noindent {\sc Fact 3}: The family $$\bigg(L(t)-N(\log
t)+\sum_{k\geq
1}\big(1-\exp(-te^{-S_{k-1}}(1-W_k))\big)\bigg)_{t\geq 1}$$ is
tight.

Since $$\me (K(t)|(W_k))=\sum_{k\geq
1}\big(1-e^{-tP_k}\big)=\sum_{k\geq
1}\big(1-\exp(-te^{-S_{k-1}}(1-W_k))\big),$$ Fact 1 and Fact 2
together imply the statement.

\noindent {\sc Fact 4}: The family $\big(Y(t)\big)_{t\geq 1}$,
where
$$Y(t):=\sum_{k\geq
1}\big(1-\exp(-te^{-S_{k-1}}(1-W_k))\big)1_{\{S_{k-1}>\log t\}},$$
is tight.

Since $1-\varphi$ is monotone and integrable on $(-\infty,0]$, it
is directly Riemann integrable on $(-\infty,0]$. Hence, by the key
renewal theorem (see Theorem 4.2 in \cite{Ath})
$$\me Y(e^t)=\me \sum_{k\geq 0}\big(1-\varphi(t-S_{k})\big)1_{\{S_k>t\}} \ \to \
\mu^{-1}\int_{[0,1]}\big(1-\psi(y)\big)y^{-1}{\rm d}y<\infty,$$
and Fact 4 follows.

Now we are ready to prove the lemma. Since
$$L(t)-L^\ast(t)=\bigg(L(t)-N(\log t)+\sum_{k\geq
1}\big(1-\exp(-te^{-S_{k-1}}(1-W_k))\big)\bigg)-Y(t),$$ the first
assertion of the lemma follows from Fact 3 and Fact 4.

In view of the inequality
\begin{eqnarray*}
-\bigg(\varphi(t-\widehat{S}_k)-\varphi(t-S_k)\bigg)1_{\{\widehat{S}_k
\leq t\}}&\leq& \varphi(t-S_k)1_{\{S_k\leq
t\}}-\varphi(t-\widehat{S}_k)1_{\{\widehat{S}_k\leq
t\}}\\&=&\varphi(t-S_k)1_{\{S_k\leq t<\widehat{S}_k\}}\\&-&
\bigg(\varphi(t-\widehat{S}_k)-\varphi(t-S_k)\bigg)1_{\{\widehat{S}_k
\leq t\}}\\&\leq& \varphi(t-S_k)1_{\{S_k\leq
t<\widehat{S}_k\}}\\&=&\varphi(t-S_k)1_{\{S_k\leq t,
\widehat{S}_0>t\}}\\&+&\varphi(t-S_k)1_{\{t-\widehat{S}_0< S_k\leq
t, \widehat{S}_0\leq t\}} \ \ \text{a.s.},
\end{eqnarray*}
to prove the second assertion it suffices to check the tightness
of
$$\mathcal{C}_1:=\bigg(\sum_{k\geq 0}\varphi(t-S_k)1_{\{S_k\leq
t<\widehat{S}_k\}}\bigg)_{t\geq 0} \ \ \text{and} \ \
\mathcal{C}_2:=\bigg(\sum_{k\geq
0}\bigg(\varphi(t-\widehat{S}_k)-\varphi(t-S_k)\bigg)1_{\{\widehat{S}_k
\leq t\}}\bigg)_{t\geq 0}.$$ Using \eqref{yyy} gives
$$\sum_{k\geq 0}\varphi(t-S_k)1_{\{t-\widehat{S}_0< S_k\leq t,\, \widehat{S}_0\leq
t\}}\leq
\varphi(0)\big(N(t)-N(t-\widehat{S}_0)\big)1_{\{\widehat{S}_0\leq
t\}}\overset{d}{\leq} \varphi(0)N(\widehat{S}_0).$$ It is clear
that $$\bigg(\sum_{k\geq 0}\varphi(t-S_k)1_{\{S_k\leq
t\}}\bigg)1_{\{\widehat{S}_0>t\}} \ \tp \ 0, \ \ t\to\infty,$$ and
the tightness of $\mathcal{C}_1$ follows. Using the mean value
theorem for differentiable functions and the monotonicity of
$\psi^\prime$ we obtain
$$\bigg(\varphi(t-\widehat{S}_k)-\varphi(t-S_k)\bigg)1_{\{\widehat{S}_k
\leq t\}}\leq
e^{t-S_k}(-\psi^\prime(e^{t-\widehat{S}_k}))1_{\{\widehat{S}_k\leq
t\}}\widehat{S}_0=-\varphi^\prime(t-\widehat{S}_k)1_{\{\widehat{S}_k\leq
t\}}\widehat{S}_0e^{\widehat{S}_0}.$$ Since $$\me\sum_{k\geq
0}(-\varphi^\prime(t-\widehat{S}_k))1_{\{\widehat{S}_k\leq
t\}}=\mu^{-1}\int_{[0,t]}(-\varphi^\prime(y)){\rm d}y \ \to \
\mu^{-1}\varphi(0), \ \ t\to\infty,$$ the family $\mathcal{C}_2$
is tight. The proof is complete.
\end{proof}
Further we need a preliminary result which establishes a weak law
of large numbers for $C(t)$.
\begin{lemma}\label{aux}
Suppose $\mu<\infty$, $\nu=\infty$, and the distribution of $|\log
W|$ is non-lattice. Then
$${C(t)\over k(t)} \ \tp \ \mu^{-1}, \ \ t\to\infty,$$
where $$k(x):=\int_0^x \varphi(y){\rm d}y, \ \ x>0.$$
\end{lemma}
\begin{proof}
The assumption $\nu=\infty$ is equivalent to $\lix k(x)=\infty$.
In view of Chebyshev's inequality it is enough to check that
$$\me C^2(t)\ \sim \ (\me C(t))^2 \ \sim \ \mu^{-2}k^2(t), \ \
t\to\infty.$$ By Lemma \ref{sg}(a), the required asymptotics of
$\me C(t)$ follows easily. Using the equality
$$C(t)=\varphi(t)+C^\prime(t-S_1)1_{\{S_1\leq t\}} \ \
\text{a.s.},$$ where $C^\prime(t):=\sum_{k\geq
1}\varphi(t-S_k+S_1)1_{\{S_k-S_1\leq t\}}\od C(t)$ is independent
of $S_1$, we have
\begin{equation}\label{1111}\me C^2(t)=2\int_{[0,\,t]} \varphi(t-x)\me C(t-x){\rm d}\me N(x)-
\int_{[0,\,t]} \varphi^2(t-x){\rm d}\me N(x).
\end{equation}
The second term exhibits the following asymptotics
\begin{equation}\label{1112}
\int_{[0,\,t]} \varphi^2(t-x){\rm d}\me N(x)=o(k(t)), \ \
t\to\infty.
\end{equation}
To see this, use the key renewal theorem in the case
$\int_{[0,\,\infty)}\varphi^2(x){\rm d}x<\infty$ or Lemma
\ref{sg}(a) followed by l'H\^{o}pital rule in the case $\lit
\int_{[0,\,t]}\varphi^2(x){\rm d}x=\infty$.

Since both $k(t)$ and $(1-\varphi(t))k(t)$ are nondecreasing
functions we apply Lemma \ref{sg}(b) to obtain, as $t\to\infty$,
\begin{equation}
\int_{[0,\,t]} \varphi(t-x)k(t-x){\rm d}\me N(x) \ \sim \
\mu^{-1}\int_{[0,\,t]}\varphi(x)k(x){\rm
d}x=(2\mu)^{-1}k^2(t).\label{1}
\end{equation}
Further, for fixed $a\in (0,t)$
\begin{equation}\label{11}
\int_{[t-a,\, t]} \varphi(t-x)k(t-x){\rm d}\me N(x)\leq
k(a)\big(\me N(t)-\me N(t-a)\big)\leq k(a)\me N(a)
\end{equation}
in view of \eqref{yyy}. Hence
$$\int_{[0,\,t]} \varphi(t-x)k(t-x){\rm d}\me N(x) \ \sim \
\int_{[0,\,t-a]} \varphi(t-x)k(t-x){\rm d}\me N(x), \ \
t\to\infty.$$ Likewise, since $$\underset{x\in [0,\,a]}{\sup}\,\me
C(x) <\infty,$$ we conclude that
\begin{equation}\label{111}
\int_{[t-a,\, t]} \varphi(t-x)\me C(t-x){\rm d}\me N(x)\leq
\underset{x\in [0,a]}{\sup}\,\me C(x)\me N(a)<\infty.
\end{equation}
Now we are ready to derive the asymptotics of $\me C^2(t)$. For
any $\varepsilon\in (0,\mu^{-1})$ there exists $x_0>0$ such that
$\mu^{-1}-\varepsilon\leq \me C(y)/k(y)\leq \mu^{-1}+\varepsilon$
for $y\geq x_0$. With this $x_0$ we have
\begin{eqnarray*}
\int_{[0,\,t]}\varphi(t-x)\me C(t-x){\rm d}\me N(x)&\leq&
(\mu^{-1}+\varepsilon)\int_{[0,\,t-x_0]}\varphi(t-x)k(t-x){\rm
d}\me N(x)\\&+& \int_{[t-x_0,\,t]}\varphi(t-x)\me C(t-x){\rm d}\me
N(x)\\&\overset{\eqref{11},\eqref{111}}{\sim}&
(\mu^{-1}+\varepsilon)\int_{[0,\,t]}\varphi(t-x)k(t-x){\rm d}\me
N(x)+O(1)\\&\overset{\eqref{1}}{\sim}&
(\mu^{-1}+\varepsilon)(2\mu)^{-1} k^2(t).
\end{eqnarray*}
Sending $\varepsilon \to 0$ and recalling \eqref{1111} and
\eqref{1112} we conclude that
$$\underset{t\to\infty}{\lim\sup}{\me C^2(t)\over k^2(t)}\leq
\mu^{-2}.$$ Arguing similarly we obtain the converse inequality
for the lower limit. The proof is complete.
\end{proof}
\begin{lemma}\label{main6}
Suppose $\mu<\infty$, $\nu=\infty$, and the distribution of $|\log
W|$ is non-lattice. Then
\begin{equation*}
{L(t)-C(\log t)\over \sqrt{\mu^{-1}k(\log t)}}\ \dod \
{\mathcal{N}}(0, 1), \ \ t\to\infty.
\end{equation*}
\end{lemma}
\begin{proof}
By Lemma \ref{au}, it is enough to prove that
\begin{equation}\label{asy}
{L^\ast(t)-C(\log t)\over \sqrt{\mu^{-1}k(\log t)}}\ \dod \
{\mathcal{N}}(0, 1), \ \ t\to\infty.
\end{equation}
Set $$X_{ti}:={\big(\exp
\big(-te^{-S_{i-1}}(1-W_i)\big)-\psi(te^{-S_{i-1}})\big)1_{\{S_{i-1}\leq
\log t\}}\over \sqrt{\mu^{-1}k(\log t)}}, \ \ i\in\mn, \ t>1,$$
and note that $\me\big(X_{ti}|(W_k)_{k\leq i-1}\big)=0$. By a
martingale central limit theorem (Corollary 3.1 in \cite{Hall}),
relation
\begin{equation}\label{asy1}
{L^\ast(n)-C(\log n)\over \sqrt{\mu^{-1}k(\log n)}}\ \dod \
{\mathcal{N}}(0, 1), \ \ n\to\infty,
\end{equation}
which is just \eqref{asy} with continuous variable $t$ replaced by
integer $n$, will hold once we can show that
\begin{equation}\label{123}
\sum_{i\geq 1}\me (X_{ni}^2|(W_k)_{k\leq i-1}) \ \tp \ 1, \ \
n\to\infty,
\end{equation}
and that, for all $\varepsilon>0$,
\begin{equation}\label{124}
\sum_{i\geq 1}\me
(X_{ni}^21_{\{|X_{ni}|>\varepsilon\}}|(W_k)_{k\leq i-1}) \ \tp \
0, \ \ n\to\infty.
\end{equation}
It suffices to establish \eqref{123}, as, in view of $|X_{ni}|\leq
1/\sqrt{\mu^{-1} k(\log n)}$, \eqref{124} will follow from it.

We have
\begin{eqnarray*}
\sum_{i\geq 1}\me (X_{e^ti}^2|(W_k)_{k\leq
i-1})&=&{\int_{[0,t]}\big(\psi(2e^{t-x})-\varphi^2(t-x)\big){\rm
d}N(x)\over \mu^{-1} k(t)}\\&=&{C(t)\over
\mu^{-1}k(t)}-{\int_{[0,t]}\big(\psi(e^{t-x})-\psi(2e^{t-x})\big){\rm
d}N(x)\over \mu^{-1} k(t)}-{\int_{[0,t]}\varphi^2(t-x){\rm
d}N(x)\over \mu^{-1} k(t)}.
\end{eqnarray*}
By Lemma \ref{aux},
$\underset{t\to\infty}{\lim}\,C(t)/(\mu^{-1}k(t))=1$ in
probability. To complete the proof of \eqref{123} one has to check
that the second and the third terms converge to zero in
probability. For the third this follows from \eqref{1112} and
Markov's inequality. The function $t\mapsto \psi(e^t)-\psi(2e^t)$
is directly Riemann integrable on $\mr$ since it is nonnegative
and integrable, and the function $t\mapsto
e^{-t}\big(\psi(e^t)-\psi(2e^t)\big)$ is nonincreasing (see, for
instance, the proof of Corollary 2.17 in \cite{Durr}). By the key
renewal theorem
\begin{eqnarray*}
\me \int_{[0,t]}\big(\psi(e^{t-x})-\psi(2e^{t-x})\big){\rm
d}N(x)&\leq& \me
\int_{[0,\infty)}\big(\psi(e^{t-x})-\psi(2e^{t-x})\big){\rm
d}N(x)\\ &\to& \mu^{-1}\me
\int_{[0,\infty]}\big(e^{-y(1-W)}-e^{-2y(1-W)}\big)y^{-1}{\rm
d}y\\&=&\mu^{-1}\log 2, \ \ t\to\infty,
\end{eqnarray*}
which proves the required result for the second term.

It remains to pass from \eqref{asy1} to \eqref{asy}. We first note
that the function $k(\log t)$ is slowly varying at $\infty$. This
follows from the equality $k(\log
t)=\int_{[1,\,t]}\psi(y)y^{-1}{\rm d}y$ and the representation
theorem for slowly varying functions (Theorem 1.3.1 in
\cite{BGT}). To prove that $\underset{t\to\infty}{\lim}{k(\log
t)\over k(\log [t])}=1$, where $[t]$ denotes the integer part of
$t$, use the slow variation of $k(\log t)$ together with the
monotonicity to conclude
$$ 1\leq{k(\log t)\over k(\log [t])}\leq {k(\log ([t]+1))\over k(\log [t])} \ \to \ 1, \ \ t\to\infty.$$
Now we intend to prove the tightness of the family
$\big(C(t)-C([t])\big)$. To this end we use the equality
$$C([t])-C(t)=\int_{[0,[t]]}\big(\varphi([t]-x)-\varphi(t-x)\big){\rm d}N(x)-\int_{[[t],t]}\varphi(t-x){\rm d}N(x).$$
By the mean value theorem
\begin{eqnarray*}
\varphi([t]-x)-\varphi(t-x)&=& -\varphi^\prime(\theta)(t-[t])\leq
e^\theta (-\psi^\prime(\theta))\\&\leq&
e^{t-x}(-\psi^\prime([t]-x))= -\varphi^\prime([t]-x)e^{t-[t]}\leq
-\varphi^\prime([t]-x)e,
\end{eqnarray*}
where $\theta$ is some value from $[[t]-x, t-x]$. Consequently,
$$C([t])-C(t)\leq e \int_{[0,[t]]}(-\varphi^\prime([t]-x)){\rm
d}N(x),$$ and the right-hand side is bounded in probability as, by
the key renewal theorem, its expectation goes to
$e\psi(1)\mu^{-1}$ (the function $t\mapsto -\varphi^\prime(t)$ is
directly Riemann integrable on $\mr^+$ since it is integrable on
$\mr^+$ and nonnegative, and $t\mapsto
e^{-t}(-\varphi^\prime(t))=-\psi^\prime(e^t)$ is a nonincreasing
function). On the other hand, in view of \eqref{yyy}
$$C([t])-C(t)\geq -\int_{[[t],t]}\varphi(t-x){\rm d}N(x)\geq -(N(t)-N([t]))\overset{d}{\geq}
-N(1).$$ Finally we want to show that the family
$\big(L^\ast(t)-L^\ast([t])\big)$ is tight. By Lemma \ref{au}, it
is enough to check the tightness of $\big(L(t)-L([t])\big)$. Since
$L(t)-L([t])$ represents the fluctuation of the number of empty
boxes after throwing $\pi_t-\pi_{[t]}$ balls, and the latter
variable is bounded from above by a Poisson variable with mean
one, the desired tightness follows. The proof of the lemma is
complete.
\end{proof}

\noindent {\sc Step 2}. The purpose of this step is investigating
convergence in distribution of $C(t)$. The cases (a), (b1) and
(c1) and the cases (b2) and (c2) are treated separately in Lemma
\ref{222} and Lemma \ref{22}, respectively.

In what follows we use the following notation. If
$\sigma^2<\infty$ we denote by $Z(\cdot)$ the Brownian motion and
set $g(t):=\sqrt{\sigma^2\mu^{-3}t}$. If condition \eqref{domain0}
holds we denote by $Z(\cdot)$ the Brownian motion and let $g(t)$
be any nondecreasing function such that $g(t)\sim \mu^{-3/2}c(t)$,
$t\to\infty$. If condition \eqref{domain1} holds we denote by
$Z(\cdot)$ the $\alpha$-stable L\'{e}vy process such that $Z(1)$
has characteristic function \eqref{st1}, and let $g(t)$ be any
nondecreasing function such that $g(t) \sim
\mu^{-1-1/\alpha}c(t)$, $t\to\infty$.

It is well-known that under either of the conditions of the
preceding paragraph, i.e. whenever the law of $|\log W|$ belongs
to the domain of attraction of an $\alpha$-stable law, $\alpha\in
(1,2]$, $${S_{[t\cdot]}-\mu (t\cdot)\over {\rm const}\,g(t)} \
\Rightarrow \ -Z(\cdot), \ \ t\to\infty$$ in $D:=D[0,1]$ under the
$J_1$ topology. While the one-dimensional convergence is a
classical result \cite{Gne}, the functional version is due to
Skorohod (Theorem 2.7 in \cite{Skor}). Since
$$\underset{u\in
[0,1]}{\sup}|\widehat{S}_{[tu]}-S_{[tu]}|=\widehat{S}_0,$$ the
same functional limit theorem as above also holds for
$\widehat{S}_{[t\cdot]}$. An appeal to Theorem 13.7.1 in
\cite{Whitt2} allows us to conclude that\footnote{According to
Theorem 1b in \cite{Bing73}, relation \eqref{16} also holds for
the non-stationary renewal process $(N(t))_{t\geq 0}$. Since our
argument imitates one given in \cite{Bing73}, we omit details.}
\begin{equation}\label{16}
W_t(\cdot):={\widehat{N}(t\cdot)-\mu^{-1}(t\cdot)\over g(t)} \
\Rightarrow \ Z(\cdot), \ \ t\to\infty,
\end{equation}
in $D$ under the $M_1$-topology. Certainly, \eqref{16} entails the
one-dimensional convergence $W_t(1)\Rightarrow Z(1)$,
$t\to\infty$. Hence, by Skorohod's representation theorem there
exist versions\newline $\bar{W}_t(1)\od W_t(1)$ and $\bar{Z}(1)\od
Z(1)$ such that
\begin{equation*}
\bar{W}_t(1) \ \to \ \bar{Z}(1), \ \ t\to\infty,
\end{equation*}
almost surely. In particular, for any $\varepsilon>0$ there exists
an a.s. finite $T>0$ such that
\begin{equation}\label{xzx}
|\bar{W}_v(1)|\leq |\bar{Z}(1)|+\varepsilon \ \ \text{for all} \ \
v\geq T.
\end{equation}
For multiple later use let us write the following estimate: for
any positive $x(t)$ such that $\lit x(t)=\infty$ and any $a>0$
\begin{equation}\label{mul}
{\bigg|\int_{[0,\,at]}(\widehat{N}(v)-\mu^{-1}v){\rm
d}(-\varphi(v))\bigg|\over x(t)}\overset{d}{\leq}
o_P(1)+(|\bar{Z}(1)|+\varepsilon){\int_{[0,\,at]}g(v){\rm
d}(-\varphi(v))\over x(t)},
\end{equation}
where $o_P(1)$ denotes a term that converges to zero in
probability, as $t\to\infty$. This can be proved as follows:
\begin{eqnarray*}
\int_{[0,\,at]}(\widehat{N}(v)-\mu^{-1}v){\rm d}(-\varphi(v))&\od&
\int_{[0,\,at]}\bar W_v(1)g(v){\rm d}(-\varphi(v))
\nonumber\\&=&\int_{[0,\,at]}\ldots
1_{\{T>at\}}+\int_{[0,\,T]}\ldots 1_{\{T\leq
at\}}+\int_{[T,\,at]}\ldots 1_{\{T\leq
at\}}\nonumber\\&=:&I_1(t)+I_2(t)+I_3(t).
\end{eqnarray*}
It is plain that $\lit I_1(t)=0$ in probability. As to $I_2(t)$,
write
$${|I_2(t)|\over x(t)}\leq {\int_{[0,\,T]}\big|\bar{W}_v(1)\big|g(v){\rm
d}(-\varphi(v))\over x(t)} \ \tp \ 0, \ \ t\to\infty.$$ Finally
$${|I_3(t)|\over x(t)}\leq
{\int_{[T,\,at]}\big|\bar{W}_v(1)\big|g(v){\rm
d}(-\varphi(v))\over x(t)}1_{\{T\leq
at\}}\overset{\eqref{xzx}}{\leq}
(|\bar{Z}(1)|+\varepsilon){\int_{[0,\,at]}g(v){\rm
d}(-\varphi(v))\over x(t)}.$$
\begin{lemma}\label{222}
Let the assumptions of parts (a) or (b1), or (c1) of Theorem
\ref{main7} hold. Then
\begin{equation}\label{56565656}
{C(\log t)-\mu^{-1}k(\log t)\over \sqrt{k(\log t)}} \ \tp \ 0, \ \
t\to\infty,
\end{equation}
and
\begin{equation}\label{565656}
{L(t)-\mu^{-1} k(\log t)\over \sqrt{\mu^{-1}k(\log t)}} \ \dod \
\mathcal{N}(0,1), \ \ t\to\infty.
\end{equation}
\end{lemma}
\begin{proof}
We start by noting that relation \eqref{565656} is an immediate
consequence of \eqref{56565656} and Lemma \ref{main6}. By Lemma
\ref{au} relation \eqref{56565656} is equivalent to
$${\widehat{C}(t)-\mu^{-1}k(t)\over\sqrt{k(t)}} \ \tp \ 0, \ \
t\to\infty.$$ To prove it, we represent the latter ratio in a more
convenient form
\begin{eqnarray*}
{\widehat{C}(t)-\mu^{-1}k(t)\over\sqrt{k(t)}}&=&{\int_{[0,t]}\varphi(t-v){\rm
d}\widehat{N}(v)-\mu^{-1}k(t)\over\sqrt{k(t)}}\\ &\od&
{\int_{[0,t]}\varphi(v){\rm
d}\widehat{N}(v)-\mu^{-1}\int_{[0,t]}\varphi(v){\rm d}v
\over\sqrt{k(t)}}\\&=& {\varphi(t)(\widehat{N}(t)-\mu^{-1}t)\over
\sqrt{k(t)}}+{\int_{[0,t]}(\widehat{N}(v)-\mu^{-1}v){\rm
d}(-\varphi(v))\over \sqrt{k(t)}}\\&=& W_t(1){g(t)\varphi(t)\over
\sqrt{k(t)}}+{\int_{[0,t]}(\widehat{N}(v)-\mu^{-1}v){\rm
d}(-\varphi(v))\over \sqrt{k(t)}}.
\end{eqnarray*}
By Lemma \ref{rela}, $\lit {g(t)\varphi(t)\over \sqrt{k(t)}}=0$.
Since, in view of \eqref{16}, $W_t(1) \dod Z(1)$, $t\to\infty$, we
have
\begin{equation*}
W_t(1){g(t)\varphi(t)\over \sqrt{k(t)}} \ \tp \ 0, \ \ t\to\infty.
\end{equation*}
Use now inequality \eqref{mul} with $a=1$ and $x(t)=\sqrt{k(t)}$.
Since, by Lemma \ref{rela}, $\lit {\int_{[0,t]}g(v){\rm
d}(-\varphi(v))\over \sqrt{k(t)}}=0$ we conclude that
$${\int_{[0,t]}(\widehat{N}(v)-\mu^{-1}v){\rm d}(-\varphi(v))\over
\sqrt{k(t)}} \ \tp \ 0, \ \ t\to\infty.$$ The proof of the lemma
is complete.
\end{proof}
\begin{rem}
Set $$m(x):=\int_{[0,\,x]}\mmp\{|\log (1-W)|>y\}{\rm d}y, \ \
x>0.$$ Lemma 5.4 in \cite{slow} proves that $|m(x)-k(x)|$ is a
bounded function. This justifies the last sentence of Theorem
\ref{main7}. Also this implies that the normalization
$\sqrt{\mu^{-1}k(t)}$ ($\sqrt{\mu^{-1}k(\log t)}$) used in Lemma
\ref{aux}, Lemma \ref{main6} and Lemma \ref{222} can be safely
replaced by $\sqrt{\mu^{-1}m(t)}$ ($\sqrt{\mu^{-1}m(\log t)}$).
\end{rem}

\begin{lemma}\label{22}
{\rm (I)} Assume that $\sigma^2=\infty$ and that
$$\int_0^x y^2 \mmp\{|\log W|\in {\rm d}y\} \ \sim \ \widetilde{\ell}(x), \ \ x\to\infty,$$ for
some $\widetilde{\ell}$ slowly varying at $\infty$. Let $c(x)$ be
any positive function such that $\lix
{x\widetilde{\ell}(c(x))\over c^2(x)}=1$. Assume further that
$$\mmp\{|\log(1-W)|>x\} \ \sim \ x^{-\beta}\ell(x), \ \
x\to\infty,$$ for some $\beta \in [0,1/2)$ and some $\ell$ slowly
varying at $\infty$. Then
\begin{equation}\label{101}
{C(t)-\mu^{-1} k(t)\over \mu^{-3/2}c(t)\varphi(t)} \ \dod \
\int_{[0,1]}v^{-\beta}{\rm d}Z(v), \ \ t\to\infty,
\end{equation}
where $(Z(v))_{v\in [0,1]}$ is the Brownian motion.\newline {\rm
(II)} Assume that $$\mmp\{|\log W|>x\} \ \sim \
x^{-\alpha}\widetilde{\ell}(x), \ \ x\to\infty$$ for some
$\alpha\in (1,2)$ and some $\widetilde{\ell}$ slowly varying at
$\infty$. Let $c(x)$ be any positive function such that $\lix
{x\widetilde{\ell}(c(x))\over c^\alpha(x)}=1$. Assume further that
$$\mmp\{|\log(1-W)>x|\} \ \sim \ x^{-\beta}\ell(x), \ \
x\to\infty,$$ for some $\beta\in [0,1/\alpha)$ and some $\ell$
slowly varying at $\infty$. Then
\begin{equation}\label{102}
{C(t)-\mu^{-1}k(t)\over \mu^{-1-1/\alpha}c(t)\varphi(t)} \ \dod \
\int_{[0,1]}v^{-\beta}{\rm d}Z(v), \ \ t\to\infty,
\end{equation}
where $(Z(v))_{v\in [0,1]}$ is the $\alpha$-stable L\'{e}vy
process such that $Z(1)$ has characteristic function \eqref{st1}.
\end{lemma}
\begin{proof}
The condition
$$\mmp\{|\log(1-W)|>x\} \ \sim \ x^{-\beta}\ell(x), \ \
x\to\infty$$ is equivalent to the following
$$\mmp\{1-W\leq x\} \ \sim \ (\log (1/x))^{-\beta}\ell(\log (1/x)), \ \
x\downarrow 0.$$ By Theorem 1.7.1' in \cite{BGT}, the latter is
equivalent to
\begin{equation}\label{818}
\varphi(t) \ \sim \ t^{-\beta}\ell(t), \ \ t\to\infty.
\end{equation}
Recalling the relation $g(t)\sim {\rm const}\,c(t)$, $t\to\infty$,
we conclude that
\begin{equation}\label{819}
\lit g(t)\varphi(t)=\infty.
\end{equation}
In view of Lemma \ref{au} it suffices to prove that
$$C^\ast(t):={\widehat{C}(t)-\mu^{-1}k(t)\over g(t)\varphi(t)} \
\dod \ \int_{[0,1]}v^{-\beta}{\rm d}Z(v), \ \ t\to\infty.$$

\noindent {\sc Case $\beta=0$}. Recalling $W_t(1) \dod Z(1)$,
$t\to\infty$ and using the equality
$$C^\ast(t)=W_t(1)+{\int_{[0,t]}(\widehat{N}(v)-\mu^{-1}v){\rm
d}(-\varphi(v))\over g(t)\varphi(t)}$$ we conclude that it remains
to check that the second term in the right-hand side converges to
zero in probability. According to \eqref{mul} (with $a=1$ and
$x(t)=g(t)\varphi(t)$), it is enough to show that
\begin{equation}\label{010}
\lit {\int_{[0,\,t]}g(v){\rm d}\big(-\varphi(v)\big)\over
g(t)\varphi(t)}=0.
\end{equation}

In view of Potter's bound (Theorem 1.5.6 in \cite{BGT}), given
$A>0$ and $\delta\in (0,1/\alpha-\beta)$ (here we take $\alpha=2$
in the case (I) of the lemma) there exists $t_0>0$ such that
\begin{equation}\label{456}
{g(tu)\over g(t)}\leq A u^{1/\alpha-\delta},
\end{equation}
whenever $t\geq t_0$, $tu\geq t_0$ and $u\leq 1$. Since
$${\int_{[t_0,t]}g(v){\rm d}(-\varphi(v))\over
g(t)\varphi(t)}\overset{\eqref{456}}{\leq} A {\int_{[t_0,\,
t]}v^{1/\alpha-\delta}{\rm d}(-\varphi(v))\over
t^{1/\alpha-\delta}\varphi(t)} \ \to \ 0, \ \ t\to\infty,$$ where
the last relation is justified by Theorem 1.6.4 in \cite{BGT}, and
$$\lit {\int_{[0,t_0]}g(v){\rm d}(-\varphi(v))\over
g(t)\varphi(t)}=0,$$ which holds in view of \eqref{819},
\eqref{010} follows.

\noindent {\sc Case $\beta\neq 0$}. We use the representation: for
any fixed $\varepsilon>0$,
\begin{eqnarray*}
C^\ast(t)&=&W_t(1)+\int_{[0,1]} W_t(v)\mu_t({\rm d}v)\\&=&
W_t(1)+\int_{[0,\,\varepsilon]}\ldots+\int_{[\varepsilon,\,1]}\ldots\\&=:&W_t(1)+I_1(\varepsilon,t)+I_2(\varepsilon,t),
\end{eqnarray*}
where the measure $\mu_t$ is defined by
$\mu_t((v,1])=\varphi(vt)/\varphi(t)$, $v\in [0,1)$.

According to \eqref{818}, as $t\to\infty$, $\mu_t$ converges
weakly on $[\varepsilon,1]$ to a measure $\mu$ defined by
$\mu((v,1]):=v^{-\beta}$. Together with \eqref{16} this entails
the convergence
$$I_2(\varepsilon, t) \ \dod \ \beta \int_{[\varepsilon,1]}Z(v) v^{-\beta-1}{\rm d}v, \ \ t\to\infty$$
by Lemma \ref{impo7}. Further one can check that
$$W_t(1)+I_2(\varepsilon, t) \ \dod \ Z(1)+ \beta \int_{[\varepsilon,1]}Z(v)v^{-\beta-1}{\rm d}v, \ \ t\to\infty.$$
According to Theorem 4.2 in \cite{Bill}, it remains to show that,
for any $\gamma>0$,
\begin{equation*}
\underset{\varepsilon\downarrow
0}{\lim}\,\underset{t\to\infty}{\lim\sup}\,\mmp\{|I_1(\varepsilon,
t)|>\gamma\}=0.
\end{equation*}
In view of inequality \eqref{mul} (with $a=\varepsilon$ and
$x(t)=g(t)\varphi(t)$) this will follow once we can prove that
\begin{equation}\label{011}
\underset{\varepsilon\downarrow
0}{\lim}\,\underset{t\to\infty}{\lim\sup}\,
{\int_{[0,\,\varepsilon t]}g(v){\rm d}(-\varphi(v))\over
g(t)\varphi(t)}=0.
\end{equation}
Since $\lit {\int_{[0,t_0]}g(v){\rm d}(-\varphi(v))\over
g(t)\varphi(t)}=0$ (use \eqref{819}) and
$${\int_{[t_0,t]}g(v){\rm d}(-\varphi(v))\over
g(t)\varphi(t)}\overset{\eqref{456}}{\leq} A {\int_{[t_0,\,
t]}v^{1/\alpha-\delta}{\rm d}(-\varphi(v))\over
t^{1/\alpha-\delta}\varphi(t)} \ \sim \ {\beta\over
1/\alpha-\beta-\delta}\varepsilon^{1/\alpha-\beta-\delta},$$ where
the last relation is justified by Theorem 1.6.4 in \cite{BGT},
\eqref{011} follows. The proof of the lemma is finished.
\end{proof}

\noindent {\sc Step 3}. The purpose of this intermediate step is
to combine results of the two previous steps into a single
statement concerning convergence in distribution of $L(t)$,
properly normalized and centered. In particular, we conclude that
under the assumptions of the theorem
\begin{equation}\label{8989}
{L(t)-b(t)\over a(t)} \ \dod \ X, \ \ t\to\infty,
\end{equation}
for $b(t):=\mu^{-1}k(\log t)$, case-dependent normalizing function
$a(t)$ and case-dependent random variable $X$. Now we identify the
functions $a(t)$ and the laws of $X$ for each case.\newline {\sc
Cases (a), (b1) and (c1)}: $X\od \mathcal{N}(0,1)$ and
$a(t)=\sqrt{\mu^{-1}k(\log t)}$. This immediately follows from
\eqref{565656}.

\noindent {\sc Case (b2)}: $X\od \mathcal{N}(0,1)$ and
$a(t)=\mu^{-3/2}c(\log t)\psi(t)$. By Lemma \ref{main6} and Lemma
\ref{22} (case $\beta=0$),
$${L(t)-C(\log t)\over \sqrt{\mu^{-1}k(\log t)}}\ \dod \ \mathcal{N}(0,1) \ \ \text{and} \ \
{C(\log t)-\mu^{-1}k(\log t)\over \mu^{-3/2}c(\log t)\psi(t)} \
\dod \ \mathcal{N}(0,1), \ \ t\to\infty,$$    According to
\eqref{818}, $\varphi(t)\sim \ell(t)$, $t\to\infty$. Therefore, as
$t\to\infty$, $k(t)\sim t\ell(t)$ (use Proposition 1.5.8 in
\cite{BGT}) and $c(t)\varphi(t)\sim t^{1/2}\ell^\ast(t)\ell(t)$
which, in view of the assumption $\lit
\ell(t)(\ell^\ast(t))^2=\infty$, implies $\lit {\sqrt{k(\log
t)}\over c(\log t)\psi(t)}=0$. Hence,
$${L(t)-\mu^{-1}k(\log t)\over \mu^{-3/2}c(\log t)\psi(t)} \
\dod \ \mathcal{N}(0,1), \ \ t\to\infty.$$

\noindent {\sc Case (c2)}: $X\od \int_{[0,1]}v^{-\beta}{\rm
d}Z(v)$ and $a(t)=\mu^{-1-1/\alpha}c(\log t)\psi(t)$. By Lemma
\ref{main6} and Lemma \ref{22},
$${L(t)-C(\log t)\over \sqrt{\mu^{-1}k(\log t)}}\ \dod \ \mathcal{N}(0,1) \ \ \text{and} \ \
{C(\log t)-\mu^{-1}k(\log t)\over \mu^{-1-1/\alpha}c(\log
t)\psi(t)} \ \dod \ \int_{[0,1]}v^{-\beta}{\rm d}Z(v), \ \
t\to\infty,$$ respectively. According to \eqref{818},
$\varphi(t)\sim t^{-\beta}\ell(t)$, $t\to\infty$. Therefore, as
$t\to\infty$, $k(t)\sim {\rm const}\,t^{1-\beta}\ell(t)$ by
Proposition 1.5.8 in \cite{BGT}, and $c(t)\varphi(t)\sim
t^{1/\alpha-\beta}\ell^\ast(t)\ell(t)$. While in the case
$\beta\in [0,2/\alpha-1)$ the relation $\lit {\sqrt{k(\log
t)}\over c(\log t)\psi(t)}=0$ holds trivially, in the case
$\beta=2/\alpha-1$ it is secured by the assumption $\lit
\ell(t)(\ell^\ast(t))^2=\infty$. Hence,
$${L(t)-\mu^{-1}k(\log t)\over \mu^{-1-1/\alpha}c(\log t)\psi(t)} \
\dod \ \int_{[0,1]}v^{-\beta}{\rm d}Z(v), \ \ t\to\infty.$$

\noindent {\sc Step 4}. Depoissonization.
Since $L(\tau_n)=L_n$, where $(\tau_n)_{n\in\mn}$ are arrival
times of $(\pi_t)$, it suffices to check that
$${L(\tau_n)-b(n)\over
a(n)} \ \dod \ X, \ \ n\to\infty.$$ In the subsequent computations
we will use arbitrary but fixed $x\in\mr$. Given such an $x$ we
will choose $n_0\in\mn$ and $t_0>0$ such that $n\pm x\sqrt{n}\geq
0$ for $n\geq n_0$ and $t\pm x\sqrt{t}\geq 0$ for $t\geq t_0$.
With this notation laid down all the inequalities or equalities
that follow will be considered either for $t\geq t_0$ or $n\geq
n_0$.

The functions $a(t)$ are slowly varying. While in the cases (b2)
and (c2) this is trivial, in the remaining cases, as has already
been mentioned, this follows from the equality $k(\log
t)=\int_{[1,\,t]}\psi(y)y^{-1}{\rm d}y$ and Theorem 1.3.1 in
\cite{BGT}. The slow variation implies that the convergence $\lit
{a(ty)\over a(t)}=1$ takes place locally uniformly in $y$. In
particular,
\begin{equation}\label{121212}
\lit {a(t\pm x\sqrt{t})\over a(t)}=1.
\end{equation}
The function $b(t)$ enjoys the following property
$$\lit \big(b(t\pm x\sqrt{t})-b(t)\big)=0$$ which entails
\begin{equation}\label{121211}
\lit {b(t\pm x\sqrt{t})-b(t)\over a(t)}=0.
\end{equation}
Indeed, $$0\leq b(t+x\sqrt{t})-b(t)=\mu^{-1}\int_{[t,
t+x\sqrt{t}]} y^{-1}\psi(y){\rm d}y\leq \mu^{-1}\psi(t)\log
(1+xt^{-1/2}) \ \sim \ \mu^{-1}\psi(t)xt^{-1/2},$$ and the
corresponding relation with 'minus' sign follows similarly.

Now \eqref{121212} and \eqref{121211} ensure that \eqref{8989} is
equivalent to
\begin{equation}\label{898989}
{L(t\pm x\sqrt{t})-b(t)\over a(t)} \ \dod \ X, \ \ t\to\infty.
\end{equation}
We will need the following observation
\begin{equation}\label{us}
{M(t+x\sqrt{t})-M(t-x\sqrt{t})\over a(t)}\ \tp \ 0, \ \
t\to\infty,
\end{equation}
where the notation $M(t)=M_{\pi_t}$ has to be recalled. Actually,
we can prove a stronger assertion
\begin{equation*}
M(t+x\sqrt{t})-M(t-x\sqrt{t}) \ \tp \ 0, \ \ t\to\infty,
\end{equation*}
as follows. Since $M(t)$ is nondecreasing it suffices to show that
the expectation of the left-hand side converges to zero. To this
end, we first prove the formula
\begin{equation}\label{245}
\me M(t)=\me \sum_{k\geq 0}\big(1-\exp(-te^{-S_k})\big)=\me
\int_{[0,\,\infty)}\big(1-\exp(-te^{-y})\big){\rm d}N(y).
\end{equation}
We use a variant of the random occupancy scheme with the random
frequencies $P_k$'s defined in the Introduction in which balls are
thrown at the arrival times of the Poisson process $(\pi_t)$. It
is clear that $M(t)=0$ on the event $\{\pi_t=0\}$ and that
$$M(t)=\inf\{k\in\mn: \pi_{k+1,\,t}+\pi_{k+2,\,t}+\ldots=0\}$$ on the event $\{\pi_t\geq 1\}$, where $\pi_{k,\,t}$ is the
number of balls (out of $\pi_t$) falling in the $k$th box. Given
$\big(P_k\big)$ $(\pi_{j,\,t})_{t\geq 0}$ is a Poisson process
with intensity $P_j$, and, for different $j$, these Poisson
processes are independent. With this at hand, it remains to write
\begin{eqnarray*}
\me \big(M(t)|(P_j)\big)&=&\sum_{k\geq
0}\mmp\{M(t)>k|(P_j)\}=1-e^{-t}+ \sum_{k\geq
1}\mmp\{\pi_{k+1,\,t}+\pi_{k+2,\,t}+\ldots\geq
1|(P_j)\}\\&=&1-e^{-t}+\sum_{k\geq
1}\big(1-\exp(-t(1-P_1-\ldots-P_k))\big)\\&=&\sum_{k\geq
0}\big(1-\exp(-te^{-S_k})\big),
\end{eqnarray*}
and \eqref{245} follows on passing to the expectation. Using
\eqref{245} we have
\begin{eqnarray*}
&& \me \bigg(M(t+x\sqrt{t})-M(t-x\sqrt{t})\bigg)\\&=&\me
\int_{[0,\,\infty)}\bigg(\exp\big(-(t-x\sqrt{t})e^{-y}\big)-\exp\big(-(t+x\sqrt{t})e^{-y}\big)\bigg){\rm
d}N(y)\\&\leq & {2x\sqrt{t}\over t-x\sqrt{t}}\me
\int_{[0,\,\infty)}\exp\big(-(t-x\sqrt{t})e^{-y}\big)(t-x\sqrt{t})e^{-y}{\rm
d}N(y)\\&\sim& 2\mu^{-1}xt^{-1/2}, \ \ t\to\infty,
\end{eqnarray*}
where the last relation follows by the key renewal theorem (the
function $t\mapsto \exp\big(t-e^t\big)$ is directly Riemann
integrable on $\mr$ since it is integrable on $\mr$ and
nonnegative, and $t\mapsto \exp(-e^t)$ is a nonincreasing
function).

Further, setting $A_n(x):=\{|\tau_n-n|>x\sqrt{n}\}$ and recalling
the notation $K(t):=K_{\pi_t}$, we have, for any $\varepsilon>0$,
\begin{eqnarray*}
\mmp\bigg\{{L(\tau_n)-L(n-x\sqrt{n})\over
a(n)}>2\varepsilon\bigg\}&=&\mmp\bigg\{{M(\tau_n)-K(\tau_n)-L(n-x\sqrt{n})\over
a(n)}>2\varepsilon\bigg\}\nonumber\\&=& \mmp\big\{\ldots
1_{A^c_n(x)}+\ldots 1_{A_n(x)}>2\varepsilon\big\}\nonumber\\&\leq&
\mmp\bigg\{{M(n+x\sqrt{n})-K(n-x\sqrt{n})-L(n-x\sqrt{n})\over
a(n)}>\varepsilon\bigg\}\nonumber\\&+&\mmp\big\{\ldots
1_{A_n(x)}>\varepsilon\big\}\nonumber\\&\leq&\mmp\bigg\{{M(n+x\sqrt{n})-M(n-x\sqrt{n})\over
a(n)}>\varepsilon\bigg\}+\mmp\big(A_n(x)\big).
\end{eqnarray*}
Hence
\begin{equation}\label{op2}
\underset{n\to\infty}{\lim\sup}\,\mmp\bigg\{{L(\tau_n)-L(n-x\sqrt{n})\over
a(n)}>2\varepsilon\bigg\}\leq
\mmp\big\{|\mathcal{N}(0,1)|>x\big\},
\end{equation}
by \eqref{us} and the central limit theorem. Since the law of $X$
is continuous, we conclude that, for any $y\in\mr$ and any
$\varepsilon>0$,
\begin{eqnarray*}
\underset{n\to\infty}{\lim\sup}\,\mmp\bigg\{{L(\tau_n)-b(n)\over
a(n)}>y\bigg\}&\leq&
\underset{n\to\infty}{\lim\sup}\,\mmp\bigg\{{L(\tau_n)-L(n-x\sqrt{n})\over
a(n)}>2\varepsilon\bigg\}\\&+&\lin\,\mmp\bigg\{{L(n-x\sqrt{n})-b(n)\over
a(n)}>y-2\varepsilon\bigg\}\\&\overset{\eqref{898989},\eqref{op2}}{\leq}&
\mmp\big\{|\mathcal{N}(0,1)|>x\big\}+\mmp\big\{X>y-2\varepsilon\big\}.
\end{eqnarray*}
Letting now $x\to\infty$ and then $\varepsilon\downarrow 0$ gives
$$\underset{n\to\infty}{\lim\sup}\,\mmp\bigg\{{L(\tau_n)-b(n)\over
a(n)}>y\bigg\}\leq \mmp\big\{X>y\big\}.$$ Arguing similarly we
infer
\begin{equation}\label{op3}
\underset{n\to\infty}{\lim\sup}\,\mmp\bigg\{{L(n+x\sqrt{n})-L(\tau_n)\over
a(n)}>2\varepsilon\bigg\}\leq \mmp\big\{|\mathcal{N}(0,1)|>x\big\}
\end{equation}
and then
\begin{eqnarray*}
\underset{n\to\infty}{\lim\inf}\,\mmp\bigg\{{L(\tau_n)-b(n)\over
a(n)}>y\bigg\}&\geq& \lin\,\mmp\bigg\{{L(n+x\sqrt{n})-b(n)\over
a(n)}>y+2\varepsilon\bigg\}\\&-&
\underset{n\to\infty}{\lim\sup}\,\mmp\bigg\{{L(n+x\sqrt{n})-L(\tau_n)\over
a(n)}>2\varepsilon\bigg\}\\&\overset{\eqref{898989},\eqref{op3}}{\geq}&
\mmp\big\{X>y+2\varepsilon\big\}-\mmp\big\{|\mathcal{N}(0,1)|>x\big\}.
\end{eqnarray*}
Letting $x\to\infty$ and then $\varepsilon\downarrow 0$ we arrive
at
$$\underset{n\to\infty}{\lim\inf}\,\mmp\bigg\{{L(\tau_n)-b(n)\over
a(n)}>y\bigg\}\geq \mmp\big\{X>y\big\}.$$ The proof of Theorem
\ref{main7} is complete.


\begin{rem}\label{en}
Here we discuss what is known in cases (b3) and (c3) introduced in
Remark \ref{en1}. \newline {\sc Case (b3)}: By Lemma \ref{main6}
and Lemma \ref{22},
\begin{equation}\label{iii}
{L(t)-C(\log t)\over \sqrt{\mu^{-1}k(\log t)}}\ \dod \ X_1 \ \
\text{and} \ \ {C(\log t)-\mu^{-1}k(\log t)\over \mu^{-3/2}c(\log
t)\psi(t)} \ \dod \ X_2, \ \ t\to\infty,
\end{equation}
respectively, where $X_1$ and $X_2$ are random variables with the
standard normal distribution. According to \eqref{818},
$\varphi(t)\sim d(\ell^\ast(t))^{-2}$, $t\to\infty$. Therefore, as
$t\to\infty$, $k(t)\sim dt(\ell^\ast(t))^{-2}$ (use Proposition
1.5.8 in \cite{BGT}) and $c(t)\varphi(t)\sim d
t^{1/2}(\ell^\ast(t))^{-1}$. Consequently, \eqref{iii} is
equivalent to $${\ell^\ast(\log t)\over \log^{1/2}
t}\big(L(t)-C(\log t)\big) \ \dod \ (d/\mu)^{1/2}X_1 \ \
\text{and} \ \ {\ell^\ast(\log t)\over \log^{1/2} t}\big(C(\log
t)-\mu^{-1}k(\log t)\big) \ \dod \ d\mu^{-3/2}X_2, \ \
t\to\infty.$$ However, we do not know whether the joint
convergence of these ratios takes place, nor do we know how
dependent the random variables $X_1$ and $X_2$ are. The same
remark concerns formula \eqref{xxx} given below.
\newline {\sc Case (c3)}: By Lemma \ref{main6} and Lemma \ref{22},
\begin{equation}\label{xy}
{L(t)-C(\log t)\over \sqrt{\mu^{-1}k(\log t)}}\ \dod \ X_1 \ \
\text{and} \ \ {C(\log t)-\mu^{-1}k(\log t)\over
\mu^{-1-1/\alpha}c(\log t)\psi(t)} \ \dod \ X_2, \ \ t\to\infty,
\end{equation}
respectively, where $X_1\od \mathcal{N}(0,1)$ and $X_2\od
\int_{[0,1]}v^{1-2/\alpha}{\rm d}Z(v)$. According to \eqref{818},
$\varphi(t)\sim dt^{1-2/\alpha}(\ell^\ast(t))^{-2}$, $t\to\infty$.
Therefore, as $t\to\infty$, $k(t)\sim
d(2-2/\alpha)^{-1}t^{2-2/\alpha}(\ell^\ast(t))^{-2}$ by
Proposition 1.5.8 in \cite{BGT}, and $c(t)\varphi(t)\sim
dt^{1-1/\alpha}(\ell^\ast(t))^{-1}$. Consequently, \eqref{xy} is
equivalent to
\begin{equation}\label{xxx}
{\ell^\ast(t)\over t^{1-1/\alpha}}\big(L(t)-C(\log t)\big) \ \dod
\ (2\mu(1-1/\alpha)/d)^{-1/2}X_1, \ {\ell^\ast(t)\over
t^{1-1/\alpha}}\big(C(\log t)-\mu^{-1}k(\log t)\big) \ \dod \
d\mu^{-1-1/\alpha}X_2.
\end{equation}

It seems that in order to settle the weak convergence issue in
these cases one has to investigate the weak convergence of
$L^\ast(t)=\sum_{k\geq
1}\exp(-te^{-S_{k-1}}(1-W_k))1_{\{S_{k-1}\leq \log t\}}$ directly,
i.e. without using the decomposition
$L^\ast(t)=\big(L^\ast(t)-C(\log t)\big)+C(\log t)$.
\end{rem}

\section{Answering a question asked in \cite{Res}}\label{re}

Let $\big(\xi_k, \eta_k\big)_{k\in\mn}$ be independent copies of a
random vector $(\xi,\eta)$ with $\xi>0$ and $\eta\geq 0$ a.s. An
arbitrary dependence between $\xi$ and $\eta$ is allowed. In what
follows we also assume that ${\tt m}:=\me \xi<\infty$, $\me
\eta=\infty$, and that the law of $\xi$ is non-lattice. Set
$$V(t):=\sum_{k\geq 1}1_{\{\widetilde{S}_{k-1}\leq t<\widetilde{S}_{k-1}+\eta_k\}}, \ \ t\geq
0,$$ where $$\widetilde{S}_0:=0, \ \
\widetilde{S}_k:=\xi_1+\ldots+\xi_k, \ \ k\in\mn.$$ Assuming that
$\xi$ and $\eta$ are independent and that
$\bar{G}(x):=\mmp\{\eta>x\}$ is regularly varying at $\infty$ with
index $-\beta$, $\beta\in[0,1)$, Proposition 3.2 in \cite{Res}
proves\footnote{Actually the cited result treats the
finite-dimensional convergence.} that
\begin{equation}\label{3232}
{V(t)-\sum_{k\geq
1}\bar{G}(t-\widetilde{S}_{k-1})1_{\{\widetilde{S}_{k-1}\leq
t\}}\over \sqrt{{\tt m}^{-1}\int_{[0,\,t]}\bar{G}(y){\rm d}y}} \
\dod \ \mathcal{N}(0,1), \ \ t\to\infty.
\end{equation}
In Problem 1 of Section 5.2 in \cite{Res} the authors ask ``when
can the random centering be replaced by a non-random centering?"
Relying on the results developed in Section \ref{ma} we can answer
this question in an extended setting where $\xi$ and $\eta$ are
not necessarily independent. In particular, the replacement is
possible, i.e.,
$${V(t)-{\tt m}^{-1}\int_{[0,\,t]}\bar{G}(y){\rm d}y\over \sqrt{{\tt
m}^{-1}\int_{[0,\,t]}\bar{G}(y){\rm d}y}} \ \dod \
\mathcal{N}(0,1), \ \ t\to\infty,$$ if either of the following
three conditions holds:
\begin{itemize}
\item $\me \xi^2<\infty$

\item $\me \xi^2=\infty$, $\int_{[0,\,x]}y^2\mmp\{\xi\in {\rm d}y\} \sim
\widetilde{\ell}(x)$, $x\to\infty$, where $\widetilde{\ell}$ is
slowly varying at $\infty$, and $\lix \bar{G}(x)c^2(x)x^{-1}=0$,
where $c(x)$ is any positive function which satisfies $\lix
{x\widetilde{\ell}(c(x))\over c^2(x)}=1$

\item $\mmp\{\xi>x\}\sim x^{-\alpha}\widetilde{\ell}(x)$,
$x\to\infty$ for some $\alpha\in (1,2)$ and some
$\widetilde{\ell}$ slowly varying at $\infty$ and $\lix
\bar{G}(x)c^2(x)x^{-1}=0$, where $c(x)$ is any positive function
which satisfies $\lix {x\widetilde{\ell}(c(x))\over
c^\alpha(x)}=1$
\end{itemize}

\noindent The replacement is not possible if either of the
following two conditions holds:
\begin{itemize}
\item $\me \xi^2=\infty$, $\int_{[0,\,x]}y^2\mmp\{\xi\in {\rm d}y\} \sim
\widetilde{\ell}(x)$, $x\to\infty$, where $\widetilde{\ell}$ is
slowly varying at $\infty$; $\bar{G}(x)\sim \ell(x)$,
$x\to\infty$, where $\ell$ is slowly varying at $\infty$, and
$\lix \bar{G}(x)c^2(x)x^{-1}=\infty$, where $c(x)$ is any positive
function which satisfies $\lix {x\widetilde{\ell}(c(x))\over
c^2(x)}=1$

\item $\mmp\{\xi>x\}\sim x^{-\alpha}\widetilde{\ell}(x)$,
$x\to\infty$, for some $\alpha\in (1,2)$ and some
$\widetilde{\ell}$ slowly varying at $\infty$; $\bar{G}(x)\sim
x^{-\beta}\ell(x)$, $x\to\infty$, for some $\beta\in [0,
2/\alpha-1]$ and some $\ell$ slowly varying at $\infty$; $\lix
\bar{G}(x)c^2(x)x^{-1}=\infty$ if $\beta=2/\alpha-1$, where
$c(x)$ is any positive function which satisfies $\lix
{x\widetilde{\ell}(c(x))\over c^\alpha(x)}=1$
\end{itemize}
In these cases $${V(t)-{\tt m}^{-1}\int_{[0,\,t]}\bar{G}(y){\rm
d}y\over {\tt m}^{-1-1/\alpha}c(t)\bar{G}(t)} \dod \ X, \ \
t\to\infty,$$ where in the first case $\alpha=2$ and $X\od
\mathcal{N}(0,1)$, and in the second case $X\od
\int_{[0,1]}v^{-\beta}{\rm d}Z(v)$, where $(Z(v))_{v\in [0,1]}$ is
an $\alpha$-stable L\'{e}vy process with characteristic function
\eqref{st1}.

To justify these statements we first note that mimicking the proof
of Lemma \ref{main6} we can check that relation \eqref{3232} holds
under the standing assumptions of this section. Let $E$ be a
random variable with the standard exponential distribution which
is independent of everything else. We claim that
\begin{equation}\label{101010}
-\int_{[1,\,\infty)}\widetilde{N}(\log x)e^{-x}{\rm d}x
\overset{d}{\leq} R(t):= \sum_{k\geq
0}\big(\bar{G}(t-\widetilde{S}_k)-\widehat{\varphi}(t-\widetilde{S}_k)\big)1_{\{\widetilde{S}_k\leq
t\}}\overset{d}{\leq} \int_{[0,\,1]}\widetilde{N}(|\log
x|)e^{-x}{\rm d}x,
\end{equation}
where
$$\widetilde{N}(t):=\inf\{k\in\mn_0: \widetilde{S}_k>t\} \ \ \text{and} \ \ \widehat{\varphi}(t):=\me \exp(-e^{t-\eta}), \ \ t\geq
0.$$ Using the subadditivity of $t\to t^+$, $t\in\mr$ and the
distributional subadditivity of $\widetilde{N}(t)$ (see
\eqref{yyy}) we obtain
\begin{eqnarray*}
\int_{[0,\,t]}\big( 1_{\{\eta>t-y\}}-1_{\{\log
E+\eta>t-y\}}\big){\rm d}\widetilde{N}(y)&=&
\widetilde{N}\big((t-\eta-\log
E)^+)-\widetilde{N}((t-\eta)^+\big)\\&\leq&
\widetilde{N}\big((t-\eta)^++ (\log
E)^-)-\widetilde{N}((t-\eta)^+\big)\\&\overset{d}{\leq}&
\widetilde{N}\big((\log E)^-\big).
\end{eqnarray*}
Hence
\begin{eqnarray*}
R(t)&=&\me_{\eta, E} \int_{[0,\,t]}\big(
1_{\{\eta>t-y\}}-1_{\{\log E+\eta>t-y\}}\big){\rm
d}\widetilde{N}(y)\\&\overset{d}{\leq}& \me_{\eta,
E}\widetilde{N}\big((\log
E)^-\big)=\int_{[0,\,1]}\widetilde{N}(|\log x|)e^{-x}{\rm d}x.
\end{eqnarray*}
The lower bound in \eqref{101010} can be proved similarly.

With \eqref{101010} at hand, we conclude that Lemma \ref{222} and
Lemma \ref{22} are still valid if $\varphi(t)$ is replaced by
$\bar{G}(t)$ and $C(t)$ is replaced by $\sum_{k\geq
1}\bar{G}(t-\widetilde{S}_{k-1})1_{\{\widetilde{S}_{k-1}\leq
t\}}$. It remains to combine these generalizations of Lemma
\ref{222} and Lemma \ref{22} and our extended version of
\eqref{3232}.

\section{Appendix}

Lemma \ref{Toeplitz1} which is our main technical tool for proving
Theorem \ref{main} is a rather particular case of a Toeplitz-
Schur theorem (see Lemma 8.1 in \cite{GIM2}). On the other hand,
this result follows immediately by an application of the Lebesgue
bounded convergence theorem.
\begin{lemma}\label{Toeplitz1}
Let $(s_n)_{n\in\mn}$ be a sequence of real numbers such that
$\lin s_n=s\in (0,\infty)$ and $(c_{n,\,m})_{n\in\mn, m\in\mn}$ an
array of nonnegative numbers which satisfy (A) $\lin c_{n,\,m}=0$,
for each $m\in\mn$, and (B) $\sum_{m=1}^{n}c_{n,\,m}=1$. Then
$\lin \sum_{m=1}^n c_{n,\,m}s_m=s$.
\end{lemma}

\begin{lemma}\label{taub}
Let $\xi$ and $\eta$ be positive random variables. The relation
$$\underset{x\downarrow 0}{\lim}\,{\mmp\{\xi\leq x\}\over \mmp\{\eta\leq x\}}=c\in [0,\infty]$$ entails $$\underset{y\to\infty}{\lim}\,
{\me e^{-y\xi}\over \me e^{-y\eta}}=c.$$
\end{lemma}
\begin{proof}
By symmetry, it suffices to consider the case $c\in [0,\infty)$.
For any $\varepsilon>0$ there exists $x_0>0$ such that
$\mmp\{\xi\leq x\}/\mmp\{\eta\leq x\}\leq c+\varepsilon$ for all
$x\in (0,x_0]$. With this $x_0$ we have

\begin{eqnarray*}
{\me e^{-y\xi}\over \me e^{-y\eta}}&\leq&
{\int_0^{x_0}e^{-yx}\mmp\{\xi\leq x\}{\rm d}x+ \int_{x_0}^\infty
e^{-yx}\mmp\{\xi\leq x\}{\rm d}x\over
\int_0^{x_0}e^{-yx}\mmp\{\eta\leq x\}{\rm d}x}\\&\leq&
{(c+\varepsilon)\int_0^{x_0}e^{-yx}\mmp\{\eta\leq x\}{\rm
d}x+y^{-1}e^{-yx_0}\over \int_0^{x_0}e^{-yx}\mmp\{\eta\leq x\}{\rm
d}x}\\&\leq& (c+\varepsilon)+{y^{-1}e^{-yx_0}\over
\int_{x_0/2}^{x_0}e^{-yx}\mmp\{\eta\leq x\}{\rm d}x}\\&\leq&
(c+\varepsilon)+{1\over \mmp\{\eta\leq x_0/2\} (e^{yx_0/2}-1)}.
\end{eqnarray*}
Sending $y\to\infty$ and then $\varepsilon\to 0$ proves
$$\underset{y\to\infty}{\lim\sup}{\me e^{-y\xi}\over \me e^{-y\eta}}\leq c.$$ The lower limit (when $c>0$) can be treated similarly.
\end{proof}
Before stating the next result we recall notation: $\varphi(t)=\me
\exp(-e^t(1-W))$, $k(t)=\int_{[0,t]}\varphi(y){\rm d}y$. The
functions $g(t)$ were defined in the paragraph preceding Lemma
\ref{222}.
\begin{lemma}\label{rela}
Assume that $\nu=\infty$ and that either conditions
\eqref{domain0} and \eqref{555} or \eqref{domain1} and
\eqref{555555} hold, or $\sigma^2<\infty$. Then
\begin{equation}\label{rel}
\lit {g(t)\varphi(t)\over \sqrt{k(t)}}=0 \ \ \text{and} \ \ \lit
{\int_{[0,t]}g(y){\rm d}(-\varphi(y))\over \sqrt{k(t)}}=0.
\end{equation}
\end{lemma}
\begin{proof}
{\sc Case $\sigma^2<\infty$}. The first relation in \eqref{rel} is
immediate:
$${g^2(t)\varphi^2(t)\over k(t)}={\rm const}\, {t\varphi^2(t)\over
k(t)}\leq {\rm const}\,\varphi(t) \ \to \ 0, \ \ t\to\infty.$$
Condition $\nu=\infty$ is equivalent to $\lit k(t)=\infty$.
Therefore if the integral $\int_{[0,\infty)}y^{1/2}{\rm
d}(-\varphi(y))$ converges the second relation in \eqref{rel}
holds trivially. Assume that $\lit \int_{[0,t]}y^{1/2}{\rm
d}(-\varphi(y))=\infty$. Integrating by parts, we have
$${1\over
\sqrt{k(t)}}\int_{[0,\,t]}y^{1/2}{\rm d}(-\varphi(y)) \ \sim \
{1\over 2\sqrt{k(t)}}\int_{[0,\,t]}\varphi(y)y^{-1/2}{\rm d}y, \ \
t\to\infty.$$ By l'H\^{o}pital rule,
\begin{equation*}
{1\over \sqrt{k(t)}}\int_{[1,\,t]}\varphi(y)y^{-1/2}{\rm d}y \
\sim \ 2\sqrt{k(t)/t} \ \to \ 0, \ \ t\to\infty,
\end{equation*}
which proves the second relation in \eqref{rel}.

\noindent {\sc Case when conditions \eqref{domain0} and
\eqref{555} hold}. Let $\eta$ be a random variable with
distribution such that
$$\mmp\{\eta\leq x\} \ \sim \ {1\over (\ell^\ast(-\log x))^2}, \ \ x\downarrow
0.$$ Then \eqref{555} is equivalent to
$$\underset{x\downarrow 0}{\lim}\,{\mmp\{1-W\leq x\}\over \mmp\{\eta\leq
x\}}=0.$$ By Lemma \ref{taub}, $$\lit {\psi(t)\over \me
e^{-t\eta}}=0.$$ Since, by Theorem 1.7.1' in \cite{BGT}, $\me
e^{-t\eta} \sim (\ell^\ast(\log t))^{-2}$, $t\to\infty$, we
conclude that
$$\lit \varphi(t)(\ell^\ast(t))^2=0.$$ Hence
$${g^2(t)\varphi^2(t)\over k(t)}\sim {\rm const}\, {t(\ell^\ast(t))^2\varphi^2(t)\over
k(t)}\leq {\rm const}\,\varphi(t)(\ell^\ast(t))^2 \ \to \ 0, \ \
t\to\infty.$$

If the integral $\int_{[0,\infty)}g(y){\rm d}(-\varphi(y))$
converges the second relation in \eqref{rel} holds trivially.
Assume that $\lit \int_{[0,t]}g(y){\rm d}(-\varphi(y))=\infty$.
According to Theorem 1.8.3 in \cite{BGT}, we can assume, without
loss of generality, that $g$ is differentiable. Then, $g^\prime(t)
\sim {\rm const}\, t^{-1/2}\ell^\ast(t)$, $t\to\infty$.
Integrating by parts, we have
$${1\over
\sqrt{k(t)}}\int_{[0,t]}g(y){\rm d}(-\varphi(y)) \ \sim \ {1\over
\sqrt{k(t)}}\int_{[1,t]}\varphi(y)g^\prime(y){\rm d}y, \ \
t\to\infty.$$ By l'H\^{o}pital rule,
\begin{equation}\label{3443}
{1\over \sqrt{k(t)}}\int_{[1,\,t]}\varphi(y)g^\prime(y){\rm d}y \
\sim \ 2g^\prime(t)\sqrt{k(t)} \ \sim \ {\rm const}\,{g(t)\over
t}\sqrt{k(t)}, \ \ t\to\infty.
\end{equation}
If $\varphi(t)\sim (\ell^\ast(t))^{-2}$, $t\to\infty$ then, by
Proposition 1.5.8 in \cite{BGT}, $\lit {g(t)\sqrt{k(t)}\over t}
=1$. Therefore, the right-hand side of \eqref{3443} goes to zero,
as $t\to\infty$, if condition \eqref{555} holds.

The case when conditions  \eqref{domain1} and \eqref{555555} hold
can be treated similarly, and we omit details.
\end{proof}

Let $(S_k^\ast)_{k\in\mn_0}$ be a zero-delayed random walk with
positive steps. Set $$N^\ast(x):=\inf\{k\in\mn_0: S_k^\ast>x\}, \
\ x\geq 0.$$ Lemma \ref{sg} is used in Section \ref{ma} for
investigating the asymptotics of moments.
\begin{lemma}\label{sg}
Suppose $\me S_1^\ast<\infty$, and the law of $S_1^\ast$ is
non-lattice.

\noindent (a) Let $r: [0,\infty)\to [0,\infty)$ be a nonincreasing
function such that
$$\lit \int_{[0,\,t]}r(y){\rm d}y=\infty.$$ Then
\begin{equation*}
\me \int_{[0,\,t]}r(t-z){\rm d}N^\ast(z) \ \sim \ (\me
S_1^\ast)^{-1}\int_{[0,\,t]}r(z){\rm d}z, \ \ t\to\infty.
\end{equation*}

\noindent (b) Let $r_1, r_2: [0,\infty)\to [0,\infty)$ be
nondecreasing functions such that $r_1(t)\geq r_2(t)$, $t\geq 0$,
and
\begin{equation}\label{sg1}
\lit \int_{[0,\,t]}\big(r_1(y)-r_2(y)\big){\rm d}y=\infty \ \
\text{and} \ \ \lit {r_1(t)+r_2(t) \over
\int_{[0,\,t]}\big(r_1(y)-r_2(y)\big){\rm d}y}=0.
\end{equation}
Then
\begin{equation*}
\me \int_{[0,\,t]}\big(r_1(t-z)-r_2(t-z)\big){\rm d}N^\ast(z) \
\sim \ (\me
S_1^\ast)^{-1}\int_{[0,\,t]}\big(r_1(z)-r_2(z)\big){\rm d}z, \ \
t\to\infty.
\end{equation*}
\end{lemma}
\begin{rem}
The conclusion of Lemma \ref{sg}(b) is in force whenever $r_1$ is
a nondecreasing function of {\it subexponential growth} satisfying
$\int_{[0,\,\infty)} r_1(y){\rm d}y=\infty$ and $r_2\equiv 0$.

Let us further note that the second condition in \eqref{sg1}
cannot be omitted. Indeed, assuming that $r_1(t)=e^t$ and
$r_2\equiv 0$ we infer $\me\int_{[0,\,t]}r_1(t-y){\rm
d}N^\ast(y)\sim (1-\me e^{-S^\ast_1})^{-1}e^t$, whereas $(\me
S_1^\ast)^{-1}\int_{[0,\,t]}r_1(y){\rm d}y\sim (\me
S_1^\ast)^{-1}e^t$.
\end{rem}
Part (a) of Lemma \ref{sg} is a fragment of Theorem 4 in
\cite{Sgib}. The proof of part (b) requires only minor
modifications and is thus omitted.

\begin{lemma}\label{impo7}
Let $0\leq a<b<\infty$. Assume that $X_t(\cdot) \Rightarrow
X(\cdot)$, as $t\to\infty$, in $D[a, b]$ in the $M_1$ topology.
Assume also that, as $t\to\infty$, $\mu_t$ converges weakly to
$\mu$ on $[a,b]$, where $(\mu_t)$ is a family of Radon measures,
and the limiting measure $\mu$ is absolutely continuous with
respect to the Lebesgue measure. Then
$$\int_{[a,b]}X_t(\cdot)\mu_t({\rm d}y) \ \dod \
\int_{[a,b]}X(\cdot)\mu({\rm d}y), \ \ t\to\infty.$$
\end{lemma}
\begin{proof}
It suffices to prove that
\begin{equation}\label{25}
\lit \int_{[a,b]}h_t(y)\mu_t({\rm d}y)= \int_{[a,b]}h(y)\mu({\rm
d}y),
\end{equation}
whenever $\lit h_t(y)=h(y)$ in $D[a,b]$ in the $M_1$ topology, for
the desired result then follows by the continuous mapping theorem.

\noindent Since $h\in D[a,b]$ the set $D_h$ of its discontinuities
is at most countable. By Lemma 12.5.1 in \cite{Whitt2},
convergence in the $M_1$ topology implies local uniform
convergence at all continuity points of the limit. Hence $E:=\{x:
\text{there exists} \ x_t \ \text{such that} \ \lit x_t= x,
\text{but} \ \lit h_t(x_t)\neq h(x)$$\}\subseteq D_h$, and we
conclude that $\mu(E)=0$. Now \eqref{25} follows from Lemma 2.1 in
\cite{broz}.
\end{proof}

\noindent {\bf Acknowledgement}. I thank the referee for a very
careful reading and many useful suggestions that led to
substantial improvements of the paper. The referee provided me
with an impressive list of my inaccuracies and oversights which
allowed me to correct them in the final version. Further thanks
are due to the Associate Editor for several valuable comments.
Finally I am grateful to Alexander Marynych for careful reading
and pointing out several oversights, in particular, in the proof
of Lemma \ref{impo7}. The present proof of Lemma \ref{6} is due to
him.

\end{document}